\documentclass{amsart}
\usepackage{template}
\usepackage[all]{xy}
\usepackage{todonotes}
\usepackage{mathtools}
\usepackage{comment}
\usepackage{pifont}

\usepackage{float}

\title{Twists, Higher Dimer Covers, and Web Duality for Grassmannian Cluster Algebras}

\author[Banaian]{Esther Banaian}
\address[Banaian]{Department of Mathematics, University of California, Riverside, CA}
\email{estherbanaian@gmail.com}

\author[Catania]{Elise Catania}
\address[Catania]{School of Mathematics, University of Minnesota, Twin Cities, Minneapolis, MN}
\email{catan042@umn.edu}

\author[Gaetz]{Christian Gaetz}
\address[Gaetz]{Department of Mathematics, University of California, Berkeley, CA}
\email{gaetz@berkeley.edu}

\author[Moore]{Miranda Moore}
\address[Moore]{School of Mathematics, University of Minnesota, Twin Cities, Minneapolis, MN}
\email{moor2340@umn.edu}

\author[Musiker]{Gregg Musiker}
\address[Musiker]{School of Mathematics, University of Minnesota, Twin Cities, Minneapolis, MN}
\email{musiker@umn.edu}

\author[Wright]{Kayla Wright}
\address[Wright]{Department of Mathematics, Johns Hopkins University, Baltimore, MD}
\email{kaylamostasisa@gmail.com}

\begin{document}
\begin{abstract}
We study a twisted version of Fraser, Lam, and Le's higher boundary measurement map, using face weights instead of edge weights, thereby providing Laurent polynomial expansions, in Pl\"ucker coordinates, for twisted web immanants for Grassmannians. In some small cases, Fraser, Lam, and Le observe a phenomenon they call ``web duality'', where web immanants coincide with web invariants, and they conjecture that this duality corresponds to transposing the standard Young tableaux that index basis webs. We show that this duality continues to hold for a large set of $\SL_3$ and $\SL_4$ webs. Combining this with our twisted higher boundary measurement map,  we recover and extend formulas of Elkin--Musiker--Wright for twists of certain cluster variables. We also provide evidence supporting conjectures of Fomin--Pylyavskyy as well as one by Cheung--Dechant--He--Heyes--Hirst--Li concerning classification of cluster variables of low Pl\"ucker degree in $\C[\Gr(3,n)]$. 
\end{abstract}
\maketitle

\tableofcontents

\section{Introduction}
Our work\footnote{An extended abstract \cite{FPSAC2025} of this work appears in the proceedings of FPSAC 2025.} compares elements in the coordinate ring of the Grassmannian and \emph{tensor invariants}. The former we denote by $\C[\Gr(k,n)]$, where $\text{Gr}(k,n)$ is the space of $k$-dimensional linear subspaces of $\mathbb{C}^n$, and $\Gr(k,n)$ denotes the affine cone over the Pl\"ucker embedding of $\text{Gr}(k,n)$. The ring $\C[\Gr(k,n)]$ is generated by the
\emph{Pl\"ucker coordinates} $\Delta_{J}$ for $J\subset[n]$ with $|J| = k$. This description gives a natural $\mathbb{N}^n$-grading on $\C[\Gr(k,n)]$, where the piece associated to $\lambda = (\lambda_1,\ldots,\lambda_n)\in\mathbb{N}^n$  is generated by products of Pl\"ucker coordinates with column $i$ represented $\lambda_i$ times. 

Let $\mathcal{W}_\lambda(\C^r) = \text{Hom}_{\SL_r}\left(\bigotimes_{i=1}^n \bigwedge^{\lambda_i} \C^r, \C\right)$ be the space of tensor invariants of multidegree $\lambda$. This space is spanned by invariants associated to \emph{$\SL_r$ webs}, certain planar graphs embedded in a disk. For $r \leq 4$, \emph{rotation-invariant} bases consisting of \emph{web invariants} are known due to Rumer--Teller--Weyl \cite{Weyl1932} ($r=2$), Kuperberg \cite{kup96} ($r = 3$) and Gaetz, Pechenik, Pfannerer, Striker, and Swanson \cite{GPPSS24} ($r = 4$).

Postnikov defined a \textit{boundary measurement map} linking Pl\"ucker coordinates to \textit{dimer covers} (also called \emph{almost perfect matchings}) on \emph{plabic graphs} \cite{postnikov2006total}.  This map associates to a \emph{network} $N$ (a plabic graph $G$ along with a choice of edge weights in $\mathbb{C}^\times$) a point $\widehat{X}(N)$ in $\Gr(k,n)$. Lam \cite{lam2015dimers}, and later Fraser--Lam--Le \cite{FLL19}, extended the boundary measurement map using $r$-dimer covers\footnote{Multi-sets of edges incident to each internal vertex (resp. boundary vertex $i$) exactly $r$ (resp. $\lambda_i$) times.} on plabic graphs. Each $r$-dimer cover of $G$ with boundary condition $\lambda \in \mathbb{N}^n$ gives rise to an $\SL_r$ web. The $r$-fold boundary measurement $\text{Web}_r(N;\lambda)$ is a linear combination of the web invariants of all $r$-dimer covers of $G$, with coefficients coming from the edge weights of $N$. When $|\lambda| \coloneqq \sum_i\lambda_i = kr$, Fraser, Lam, and Le used $\text{Web}_r(N;\lambda)$ to define the \emph{immanant map}, an isomorphism $\mathcal{W}_\lambda(\C^r)^* \to \C[\Gr(k,n)]_\lambda$. Equivalently, there is a natural perfect pairing $\langle \cdot , \cdot \rangle: \mathcal{W}_\lambda(\C^r) \otimes \C[\Gr(k,n)]_\lambda \to \C$ which we call the \emph{FLL pairing}.

The ring $\C[\Gr(k,n)]$ is one of the most fundamental examples of a \emph{cluster algebra} \cite{scott06}. A plabic graph $G$ determines an initial seed of this cluster algebra, consisting of one Pl\"ucker coordinate for each face of $G$. By \cite{marsh2016twists, muller2017twist}, the boundary measurement map can be used to express the images of the Pl\"ucker coordinates under an important automorphism, the \textit{twist map} $\tau$ of \cite{BFZ}, as Laurent polynomials in the initial seed. Elkin--Musiker--Wright \cite{EMW23} showed that one may moreover express the images of more complicated cluster variables of Pl\"{u}cker degree two or three in terms of double and triple dimer covers using certain \emph{face weights} in place of edge weights. In our first main result, Theorem~\ref{thm:TwistGivenBySumOfPairings}, we give a wide generalization of these results by expressing twists $\tau(f)$ of arbitrary elements $f \in \mathbb{C}[\Gr(k,n)]$ as Laurent polynomials in the initial seed, using higher dimer covers, the FLL pairing, and face weights.  When $\tau(f)$ is a cluster variable up to multiplication by frozens (in particular when $f$ is a cluster variable \cite[Prop. 8.10]{marsh2016twists}), this realizes the \emph{Laurent phenomenon}. These results have a natural interpretation in terms of $\text{Web}_r^\tau(N;\lambda)$, a twisted version of Fraser--Lam--Le's higher boundary measurement map that we introduce (see Corollary~\ref{cor:twistedtensorpairing}).

In Section \ref{sec:ComputingDuals} we explain the interaction between the influential conjectures of Fomin--Pylyavskyy \cite{FP16}, the FLL pairing, and Theorem~\ref{thm:TwistGivenBySumOfPairings}. In agreement with the Fomin--Pylyavskyy conjectures, all known $\text{Gr}(k,n)$ cluster variables are web invariants, and moreover lie in the rotation-invariant web bases \cite{kup96,fraser23,GPPSS24} in the cases in which these are known to exist. In the smallest of these cases the \emph{web immanants}, dual under the FLL pairing to basis webs, themselves lie \cite{FLL19, fraser23} in the appropriate web basis. In light of the results of Section~\ref{sec:TwistsAndDimers}, these phenomena explain the especially simple expressions \cite{marsh2016twists, muller2017twist, EMW23} for twists of low Pl\"{u}cker degree cluster variables, which amount to extracting the coefficients of particular basis webs in $\text{Web}_r^\tau(N;\lambda)$. It is therefore natural to study these phenomena in the cases which are newly accessible since the discovery in \cite{GPPSS24} of an $\SL_4$ web basis.

Suppose for simplicity that $n=kr$ and $\lambda=(1,\ldots,1)\in\mathbb{N}^n$. In this setting, we have $\C[\Gr(k,n)]_\lambda\iso \W_\lambda(\C^k)$, so the FLL pairing yields a duality between $\W_\lambda(\C^r)$ and $\W_\lambda(\C^k).$ Web immanants were computed in \cite{fraser23} for $k=2$ and in \cite[Appendix]{FLL19} for $(k,n)=(3,6)$ and $(3,9)$ to be themselves web invariants. In Theorem \ref{thm:DualityComputation}, we extend these results by computing the dual basis to Kuperberg's $\SL_3$ web basis when $(k,n)=(3,12)$. Basis webs for $\W_\lambda(\C^3)$ (resp.\ $\W_\lambda(\C^4)$) are in bijection with standard Young tableaux of rectangular shape $3\times 4$ (resp.\ $4\times 3$), and our results show that in all but one case, the web immanants are again basis web invariants whose tableau is the transpose of the tableau indexing the original basis web. This duality is depicted in Tables \ref{tablep1} and \ref{tablep2}. As an application, we obtain combinatorial expansion formulas for twists of elements of $\C[\Gr(3,12)]_\lambda$ and $\C[\Gr(4,12)]_\lambda$ (Theorem \ref{thm:twist}), extending results from \cite{EMW23}. 

In our analysis of the Pl\"{u}cker degree 4 web invariants in $\C[\Gr(3,n)]$, we verify that the Fomin--Pylyavskyy conjecture is consistent with enumerative conjectures of \cite{CDHHHL22}. This is described in Section \ref{sec:Counting}.


\section{Preliminaries}\label{sec:preliminaries}

In this section, we review the necessary preliminaries to discuss our work. We begin by recalling the cluster algebra structure on the Grassmannian and reviewing the combinatorics of plabic graphs. This allows us to also state previous results expressing the twist map in terms of dimer partition functions. We then discuss the representation theory and combinatorics of $\SL_3$ and $\SL_4$ webs. This subsection will follow \cite{FP16} and review the relevant conjectures that inspire much of our work. We end this section by discussing the associated tableaux combinatorics related to webs that we will need throughout the paper. 

\subsection{Grassmannian cluster algebras}\label{subsec:gr_clusteralgs}

Let $[n] \coloneqq \{1,2,\ldots,n\}$. For $0 \leq k \leq n$, let $\binom{[n]}{k}$ denote the set of subsets of $[n]$ of size $k$. Each element of the Grassmannian $\text{Gr}(k,n)$ is represented by a full-rank $k\times n$ matrix $M$, modulo the left action of $\SL_k$. For each $k$-subset $I\in\tbinom{[n]}{k},$ the \emph{Pl\"ucker coordinate} $\plu I$ is the minor of $M$ with columns indexed by $I$. The $\plu I$ are projective coordinates defining the Pl\"{u}cker embedding $\text{Gr}(k,n)\hookrightarrow\mathbb{P}^{\binom{n}{k}-1}$. We often find it convenient to work with the affine cone $\Gr(k,n)$, where the $\Delta_I$ are genuine functions generating the coordinate ring $\C[\Gr(k,n)]$. 
Given $I \subseteq [n]$, we write $\delta^I \in \N^n$ for the associated indicator vector.  The ring $\C[\Gr(k,n)]$ is naturally $\mathbb{N}^n$-graded, with $\Delta_I$ having degree given by $\delta^I$. 

\begin{definition}[\cite{postnikov2006total}]
A \emph{plabic graph} is a planar graph embedded in a disk, with vertices properly\footnote{Plabic graphs are often allowed to have non-proper vertex colorings, but can always be made bipartite by applying contraction moves.} colored black or white. The boundary vertices are all colored black, are labeled by $1,2,\ldots,n$ in clockwise order, and are each incident to exactly one edge. A plabic graph $G$ has type $(k,n)$ if the number of internal white vertices minus the number of internal black vertices is $k$. For example, the claw graph consisting of a single internal white vertex attached to every boundary vertex is of type $(1,n)$.

The \emph{trip permutation} $\pi$ of $G$ is defined as follows. For each boundary vertex $i$, define a \emph{trip} starting at the boundary vertex $i$ and following the edges of the graph by turning maximally left at white vertices and maximally right at black vertices, until reaching another boundary vertex, defined to be $\pi(i).$ A plabic graph is \emph{reduced} if it attains the minimum number of faces among all plabic graphs with the same trip permutation. If $G$ is reduced, the trips do not self-intersect, and therefore divide the disk into left and right halves. In this case, we label each face $\mathbf{f}$ of $G$ with a $k$-set $I_\mathbf{f}\in\tbinom{[n]}{k}$, where $i\in I_\mathbf{f}$ if and only if $\mathbf{f}$ lies to the left of the trip ending at $i$. We say that a plabic graph $G$ of type $(k,n)$ is \emph{top cell} if the trip permutation is $i\mapsto (i+k)\mod n$. 
\end{definition}

\begin{figure}
\includegraphics[width=.4\textwidth]{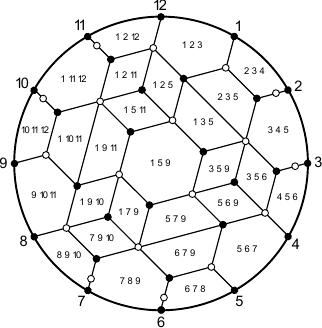}
    \caption{A top cell plabic graph of type $(3,12)$.}
    \label{fig:top cell plabic graph}
\end{figure}

\emph{Cluster algebras} \cite{FZ02} are commutative rings $A$ with distinguished generators called \emph{cluster variables}, grouped into \emph{clusters}, which are defined in a recursive way from a \emph{seed} consisting of a quiver $Q$ and a collection $\mathbf{x}$ of algebraically independent elements of $A$. The Grassmannian coordinate ring $\C[\Gr(k,n)]$ has the structure of a cluster algebra \cite{scott06}. Any reduced top cell plabic graph $G$ determines a seed for $\C[\Gr(k,n)]$ with $Q$ planar dual to $G$ and $\mathbf{x}=(\Delta_{I_\mathbf{f}})_{\text{$\mathbf{f}$ a face of $G$}}$. The Pl\"ucker coordinates for the $n$ cyclically consecutive $k$-sets are \emph{frozen variables} and are contained in every cluster. The rest of the Pl\"ucker coordinates are examples of \emph{mutable} cluster variables.

\subsection{The boundary measurement map}

For $G$ a plabic graph, an \emph{$r$-dimer cover} $D$ is a multiset of edges of $G$ such that each internal vertex is incident to exactly $r$ edges from $D$, counted with multiplicity. For $\lambda \in \N^n$, we write $\D_{r,\lambda}(G)$ for the set of all $r$-dimer covers $D$ of $G$ such that, for each $i\in[n]$, $D$ contains the edge incident to boundary vertex $i$ with multiplicity $\lambda_i$. 

We often say just \emph{dimer cover} for a $1$-dimer cover; these are central in Postnikov's boundary measurement map. When $G$ has type $(k,n)$, all dimer covers use exactly $k$ of the boundary vertices. A \emph{network} $N$ is a plabic graph together with a weight $\wt(\mathbf e) \in \mathbb{C}^{\times}$ assigned to each edge $\mathbf e$. We say that $N$ is a \emph{$(k,n)$-network} if its underlying graph is a reduced top cell plabic graph of type $(k,n)$. An  $r$-dimer cover $D$ of the underlying plabic graph of a network $N$ has \emph{edge weight}
\begin{equation} \label{eq:eweight} \ewt_N(D)\coloneqq \prod_{\mathbf{e} \in D} \wt(\mathbf{e}).\end{equation}

\begin{definition}[\cite{postnikov2006total}]
The \emph{boundary measurement map} $\widehat X$ associates to each $(k,n)$-network $N$ whose underlying reduced top cell plabic graph is $G$ a point $\widehat X(N)\in\Gr(k,n)$, given by
    \[\widehat X(N)\coloneqq\left(\plu I(N)\ \middle|\ I\in\tbinom{[n]}{k}\right) \in\Gr(k,n) \subset\C^{\tbinom{n}{k}},\]
    where \[\plu I(N)\coloneqq \sum_{D\in \mathcal{D}_{1,\delta^I}(G)}\ewt_N(D).\]
\end{definition}

It is a nontrivial fact (see \cite[Theorem 1.1]{postnikov2006total}, \cite[Corollary 2.7]{PSW}, and \cite[Theorem 4]{lam2015dimers})
that $\widehat X(N)$ in fact lies inside the image of the Pl\"{u}cker embedding of $\Gr(k,n)$, hence our usage of function $\Delta_I$ without abuse of notation.

\subsection{Web invariants}

For $\lambda\in \N^n$ we will consider two kinds of spaces of $\SL_r(\C)$ tensor invariants:
\begin{align*}
    \W_\lambda(\C^r) &\coloneqq\Hom_{\SL_r}\left(\bigotimes_{i=1}^n \bigwedge^{\lambda_i} (\C^r), \C\right), \\
    \C[\Gr(r,n)]_\lambda &\iso \Hom_{\SL_r}\left(\bigotimes_{i=1}^n \Sym^{\lambda_i} (\C^r), \C\right).
\end{align*}

\emph{Webs} were introduced by Kuperberg \cite{kup96} in order to provide a diagrammatic calculus for such spaces of $\SL_r$ tensor invariants. Our conventions for webs most closely follow those introduced by Fraser--Lam--Le \cite{FLL19}.

\begin{definition}
An \emph{$\SL_r$ tensor diagram} is a properly vertex-bicolored graph $X$, embedded in a disk with fixed cyclic edge ordering around each vertex, with $n$ black\footnote{There is a more general class of tensor diagrams where the boundary vertices may be black or white, see \cite{FP16}, but we will not consider those here.} vertices on the boundary circle, labeled clockwise by $1,2,\ldots,n$, and with edge multiplicities $m(\mathbf{e}) \geq 1$, satisfying:
\begin{enumerate}
\item All edge crossings are transverse.
\item \label{item:multiplicity-r} For all internal vertices $v$, we have $\sum_{v \in \mathbf{e}} m(\mathbf{e})=r$.
\end{enumerate}
A \emph{web} is a tensor diagram whose embedding is planar. When an edge has multiplicity $>1$, we depict it as an \emph{hourglass edge}, following the conventions of \cite{GPPSS24} (see Figure \ref{fig:SL_5 webs}).

A tensor diagram has \emph{degree} $\lambda=(\lambda_1,\ldots,\lambda_n)\in\N^n$ if, for all $1\leq i \leq n$, $\lambda_i$ is the number of edges incident to boundary vertex $i$ (counted with multiplicity). For a $\SL_r$ tensor diagram of degree $\lambda$, we always have that $|\lambda|$ is divisible by $r$; we call $|\lambda|/r$ the \emph{Pl\"{u}cker degree}. 

We say a tensor diagram is \emph{semistandard} if all edges $\mathbf{e}$ incident to the boundary have $m(\mathbf{e})=1$. We say it is \emph{dual semistandard} if there is at most one edge incident to each boundary vertex, and \emph{standard} if it is both semistandard and dual semistandard. Note that a tensor diagram is standard if and only if $\lambda\in\{0,1\}^n.$
\end{definition}

\begin{figure}
\begin{tabular}{ccc}
\includegraphics[scale=.9]{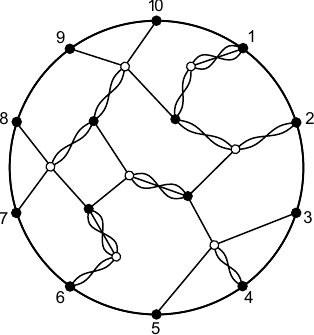}
&
\includegraphics[scale=.9]{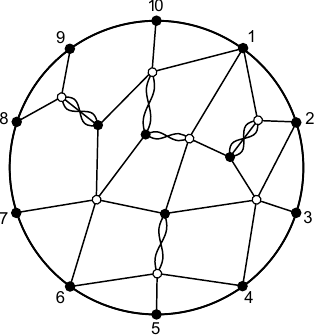}
&
\includegraphics[scale=.9]{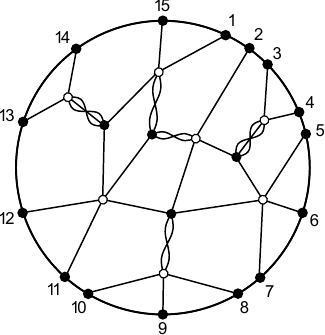}\\
$W$&$X$&$\overline X$
\end{tabular}

\caption{A dual semistandard $\SL_5$ web $W$ of degree $\lambda$, a semistandard $\SL_5$ web $X$ of degree $\lambda$, and the unclasped web $\overline{X}$, where $\lambda = (3,2,1,2,1,2,1,1,1,1)\in\N^{10}$.}
\label{fig:SL_5 webs}
\end{figure}

We now explain how $\SL_r$ tensor diagrams represent tensor invariants. Our conventions here are consistent with those of \cite{FLL19,GPPSS24}.

\begin{definition} \label{defn:dualss}
    Let $\lambda\in\N^n$ with $|\lambda|=kr,$ and let $W$ be a dual semistandard $\SL_r$ web of degree $\lambda$. Let $\mathcal{S} = (S(1),\ldots,S(n))$ be a list of subsets $S(i) \subseteq [r]$ such that $|S(i)|=\lambda_i$ for all $i$. From each such $\S$, we can define a basis element $E_\S\in\bigotimes_{i=1}^n \bigwedge^{\lambda_i} (\C^r)$ as follows. Let $s_{i\ell}$ denote the $\ell$th element of set $S(i)$, with each set $S(i)$ sorted in increasing order. Then, we define
    \[E_\mathcal{S}\coloneqq(e_{s_{11}}\wedge \cdots\wedge e_{s_{1\lambda_1}})\otimes(e_{s_{21}}\wedge \cdots\wedge e_{s_{2\lambda_2}})
    \otimes\cdots\otimes
    (e_{s_{n1}}\wedge \cdots\wedge e_{s_{n\lambda_n}}),\]
    where $e_j$ denotes the $j$th standard basis element of $\C^r$.
    
    We define the \emph{word} of $\mathcal{S}$ by reading off the indices of the basis vectors in $E_\mathcal{S}$ from left to right: 
    \[w(\mathcal{S})\coloneqq
    s_{11} \cdots s_{1\lambda_1}s_{21}\cdots s_{2\lambda_2}\cdots
    s_{n1}\cdots s_{n\lambda_n}
    \in [r]^{kr}.\]
    Define $\sign(\mathcal{S})\in\pm1$ according to the parity of the number of inversions in $w(\mathcal{S})$.

    For a dual semistandard $\SL_r$ web $W$, a \emph{consistent $r$-labeling} $\L$ of $W$ with \emph{boundary label condition} $\mathcal{S}$ is an assignment of a subset $L(\mathbf{e})\subseteq [r]$ to each edge $\mathbf{e}$ of $W$ such that:
    \begin{enumerate}
        \item $|L(\mathbf{e})|=m(\mathbf{e})$;
        \item $\bigcup_{v \in \mathbf{e}} L(\mathbf{e}) = [r]$ for each internal vertex $v$; and
        \item $L(\mathbf{e}_i)=S(i)$ for each boundary vertex $i$, where $\mathbf{e}_i$ is the incident edge.
    \end{enumerate}
    We denote by $A(\mathcal{S};W)$ the set of consistent labelings of $W$ with  boundary label condition $\mathcal{S}$, and by $a(\mathcal{S};W)$ the number of such labelings.

    The \emph{word} of a web $W$, denoted $w(W)$, is the lexicographically smallest word $w(\mathcal{S}) \in[r]^{kr}$ such that $a(\mathcal{S};W)> 0$. Define $\mathcal{S}_W=(S_W(1),\ldots,S_W(n))$ to be the boundary label condition such that $w(\mathcal{S}_W)=w(W)$, and define $\sign(W) \coloneqq \sign(\mathcal{S}_W)$.
\end{definition}

\begin{remark}
If $\lambda\in\N^n$ with $|\lambda|=kr$, and $W$ is a dual semistandard web of degree $\lambda$, then a necessary condition in order to have $a(\mathcal{S};W)>0$ is that 
\begin{equation}\label{eq:multiset_condition}
\bigsqcup_i S(i) = \{1^k,\ldots,r^k\}\quad\text{as a multiset.}
\end{equation} 
That is, each label appears exactly $k$ times on the boundary. In particular, this means that $\mathcal{S}_W$ satisfies condition \eqref{eq:multiset_condition}. 
\end{remark}

\begin{definition}\label{def:unclasping}
Let $\lambda\in\N^n$ with $|\lambda|=kr,$ and let $X$ be a semistandard $\SL_r$ web of degree $\lambda$. Define $\overline X$ to be the \emph{unclasping} of $X$, where each boundary vertex $i$ of $X$ which is incident to $\lambda_i$ edges has been separated into $\lambda_i$ boundary vertices, each incident to a single edge, while preserving planarity. Hence $\overline{X}$ is a standard $\SL_r$ web\footnote{Since $\overline{X}$ is standard, it is also dual semistandard and Definition \ref{defn:dualss} applies.} with $kr$ boundary vertices. The word of $X$ is defined by $w(X)\coloneqq w(\overline X)$ and $\sign(X)\coloneqq\sign(\overline X)$. 

Define $\mathcal{S}_X \coloneqq (S_X(1),\ldots,S_X(n))$ by $ S_X(i)=\bigsqcup_\ell S_{\overline X}(\ell),$ where $\ell$ runs over the $\lambda_i$ boundary vertices of $\overline X$ which came from unclasping vertex $i$ of $X$. Note that the boundary label sets $S_X(i)$ may be multisets.
\end{definition}

\begin{example}\label{ex:boundary-subsets}
    Let $W$ and $X$ be the dual semistandard and semistandard $\SL_5$ webs from Figure \ref{fig:SL_5 webs}. Both have degree $\lambda=(3,2,1,2,1,2,1,1,1,1)$, and $|\lambda|=15$. We have:
    \begin{align*}
        \mathcal{S}_W&=(\{1,2,3\},\{1,4\},\{1\},\{2,3\},\{5\},\{2,4\},\{3\},\{5\},\{4\},\{5\}),\\
        w(W)&=123141235243545,\\
        \mathcal S_{\overline X}&=(\{1\},\{1\},\{1\},\{2\},\{2\},\{3\},\{4\},\{2\},\{3\},\{5\},\{3\},\{4\},\{4\},\{5\},\{5\}),\\
        w(X)&=w(\overline X)=111223423534455,\\
        \mathcal{S}_X&=(\{1,1,1\},\{2,2\},\{3\},\{2,4\},\{3\},\{3,5\},\{4\},\{4\},\{5\},\{5\}).
    \end{align*}
    Observe that $\mathcal{S}_W,\mathcal S_{\overline X},\mathcal{S}_X$ satisfy condition \eqref{eq:multiset_condition} (since each digit appears exactly $\frac{15}{5}=3$ times), and $a(\mathcal S_W; W)=a(\mathcal{S}_{\overline X};\overline{X})=1$ (see also Theorem \ref{thm:asw1}). Figure \ref{fig:SL_5 webs colored} shows the unique consistent labelings $\mathcal{L}_W\in A(\mathcal S_W; W)$ and $\mathcal{L}_{\overline{X}}\in A(\mathcal S_{\overline{X}}; \overline{X})$.
\end{example}

\begin{figure}
\begin{tabular}{ccc}
\adjustbox{valign=c}{\includegraphics[scale=.9]{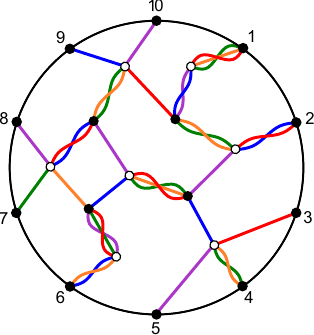}}
&
\adjustbox{valign=c}{\includegraphics[scale=.9]{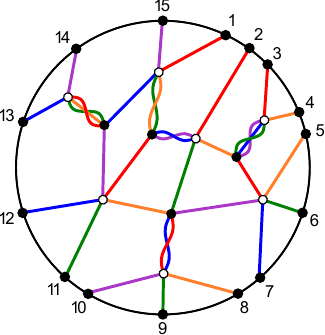}}
&
\begin{tabular}{|c|}
\hline
Key\\
1 {\color[HTML]{FF0000} \rule[2pt]{2cm}{1.5pt}}\\
2 {\color[HTML]{FF7F2A} \rule[2pt]{2cm}{1.5pt}}\\
3 {\color[HTML]{008000} \rule[2pt]{2cm}{1.5pt}}\\
4 {\color[HTML]{0000FF} \rule[2pt]{2cm}{1.5pt}}\\
5 {\color[HTML]{AB37C8} \rule[2pt]{2cm}{1.5pt}}\\
\hline
\end{tabular}
\\
\adjustbox{valign=b}{$\L_W$}&\adjustbox{valign=b}{$\L_{\overline X}$}
\end{tabular}
\caption{The unique consistent labelings $\L_W\in A(\S_W; W)$ and $\L_{\overline{X}}\in A(\S_{\overline{X}}; \overline{X})$ (see Example \ref{ex:boundary-subsets}). For each edge $\mathbf{e},$ the subset $L(\mathbf{e})\subset[5]$ is denoted by the subset of colors appearing on that edge.}
\label{fig:SL_5 webs colored}
\end{figure}

The following characterization of web invariants is a consequence of \cite[Lemma 5.4]{FLL19}, but we can take it as a definition. Applying a web invariant to a basis element of $\bigotimes_i \bigwedge^{\lambda_i}(\C^r)$ or $\bigotimes_i \Sym^{\lambda_i}(\C^r)$ amounts to counting consistent labelings of the web, with a sign which depends on the basis element. Note that this definition only works for webs; it does not apply for non-planar tensor diagrams. Throughout this paper we will use $W$ to denote a web as a graph, and $\mathbf W$ to denote the corresponding web invariant as a function.

\begin{definition}[\cite{FLL19}]\label{def:W_lambdainvariant} 
Let $W$ be a dual semistandard $\SL_r$ web of degree $\lambda$. The \emph{web invariant} $\mathbf{W}\in \mathcal{W}_\lambda(\C^r)$ is defined by its action on basis elements $E_\mathcal{S}$:
\[\mathbf{W}(E_\mathcal{S})\coloneqq \sign(\mathcal{S})a(\mathcal{S};W).\]
Such web invariants span $\mathcal{W}_\lambda(\C^r)$. 
\end{definition}

Now we consider semistandard webs, whose invariants lie in $\Hom_{\SL_r}\left(\bigotimes_{i=1}^n \Sym^{\lambda_i} (\C^r), \C\right) \iso \C[\Gr(r,n)]_\lambda$.

\begin{definition}[\cite{FP16}]
Let $X$ be a semistandard $\SL_r$ web of degree $\lambda$, with $|\lambda|=kr$. The action of the web invariant $\mathbf{X}\in\C[\Gr(r,n)]_\lambda$ on an $n$-tuple of vectors $v_i\in\C^r$ is defined by: 
    \[ \mathbf{X}(v_1,\ldots,v_n) \coloneqq \mathbf{\overline X}(\underbrace{v_1\otimes\cdots\otimes v_1}_{\lambda_1\text{~times}}\otimes\underbrace{v_2\otimes\cdots\otimes v_2}_{\lambda_2\text{~times}}\otimes\cdots\otimes\underbrace{v_n\otimes\cdots\otimes v_n}_{\lambda_n\text{~times}}),\]
where $\mathbf{\overline{X}}\in \W_{(1^{kr})}(\C^r)$ with is the unclasped web invariant.
\end{definition}

Note that when $\lambda=(1^n)$ the spaces $\W_\lambda(\C^r)$ and $\C[\Gr(r,n)]_\lambda$ are the same.

\subsubsection{Skein relations on tensor diagrams}

In order to prove Theorem \ref{thm:DualityComputation} (exhibiting a pairing between $\C[\Gr(3,12)]_{(1^{12})}$ and $\C[\Gr(4,12)]_{(1^{12})}$), we will need to expand standard $\SL_3$ web invariants in terms of Pl\"ucker coordinates. For these computations we will find it convenient to convert to the conventions of \cite{FP16} for associating an invariant to a (not necessarily planar) $\SL_3$ tensor diagram. 

\begin{definition}[Fomin--Pylyavskyy(FP) sign convention, \cite{FP16}] 
\label{def:FP_SL3_tensor_invariant}
Let $T$ be a standard $\SL_3$ tensor diagram. Applying Definition \ref{defn:dualss} to the non-planar case, we let $\mathcal{S}\in[3]^n$ be a boundary label condition, and let $\mathcal{L}\in A(\mathcal{S};T)$ be a consistent labeling of $T$. For each internal vertex $v$ of $T$, define $\sign^\text{FP}(\mathcal{L},v)$ to be the sign of the cyclic permutation of the labels $\{1,2,3\}$ around $v$, where we assign positive sign to clockwise orientation. Define $\sign^\text{FP}(\mathcal{L})\coloneqq\prod_{v}\sign^\text{FP}(\mathcal{L},v)$, where the product is over all internal vertices $v$ of $T$.

The \emph{FP tensor invariant} $\mathbf{T}^{\text{FP}}\in\W_\lambda(\C^3)$, for $\lambda\in\{0,1\}^n$, is defined by its action on basis elements $E_\mathcal{S}$ (compare with \cite[Equation 4.1]{FP16}):
\[\mathbf{T}^{\text{FP}}(E_\mathcal{S})\coloneqq \sum_{\L\in A(\mathcal{S};T)}\sign^\text{FP}(\mathcal{L}).\]
\end{definition}

When $W$ is a \emph{web}, the FP sign convention agrees with Definition \ref{def:W_lambdainvariant} up to possibly a global sign. Here we focus only on the setting where we will use this result.

\begin{proposition}\label{prop:sign conversion}
    For each standard $\SL_3$ web $W$ of degree $(1^{12})$, we have 
    \[ \mathbf{W}^{\text{FP}}=\sign(W)\mathbf{W}.\]
\end{proposition}

\begin{proof}
By the analysis of signs in \cite[Lemma 5.4]{FLL19}, we know that these two invariants agree up to a global sign. By construction, we have $\mathbf{W}(E_{\mathcal{S}_W})=\sign(W)a(\mathcal{S}_W;W)$. Therefore, the claim will follow after we compute $\mathbf{W}^\text{FP}(E_{\mathcal{S}_W})>0$ for all standard webs $W$ with exactly 12 boundary vertices. This is verified by direct computation.
\end{proof}

\begin{remark}\label{rem:props of FP signs}
The FP sign convention (Definition \ref{def:FP_SL3_tensor_invariant}) has the following advantageous properties:
\begin{itemize}
\item The local relation shown in Figure \ref{fig:sl3_wrench} (the \emph{wrench relation}), allows us to write a web invariant as a linear combination of invariants from tensor diagrams with fewer internal vertices.
\item Superimposition of tensor diagrams is multiplicative: if a tensor diagram is the union of two components, i.e.\ $T=R\cup S$ where $R$ and $S$ are disjoint as graphs, 
then $\mathbf{T}^\text{FP}=\mathbf{R}^\text{FP}\cdot \mathbf{S}^\text{FP}.$
\item FP tensor invariants behave well with respect to rotation of tensor diagrams. Let $\rho(\mathbf{T})$ denote the invariant obtained from rotating a tensor diagram $T$ counterclockwise by $2\pi/n$, and let $\rho(\S)=(S(2),\ldots,S(n),S(1))$. Then
$\mathbf{T}^\text{FP}(E_\mathcal{S})=\rho(\mathbf{T}^\text{FP})(E_{\rho(\mathcal{S})}).$

\end{itemize} 
\end{remark}

\begin{figure}[h]
\[
\adjustbox{valign=c}{\includegraphics[]{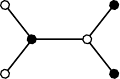}}\quad=\quad\adjustbox{valign=c}{\includegraphics[]{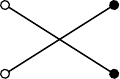}}\quad-\quad\adjustbox{valign=c}{\includegraphics[]{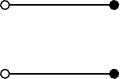}}
\]
\caption{The wrench relation for $\SL_3$ tensor invariants in the FP sign convention.}
\label{fig:sl3_wrench}
\end{figure}

For extensive lists of relations between $\SL_r$ tensor diagrams, we refer the reader to \cite{FP16} ($r=3$) and \cite{GPPSS24} ($r=4$).

\subsection{Web bases and tableaux}

While $\W_\lambda(\C^r)$ and $\C[\Gr(r,n)]_\lambda$ are spanned by web invariants, the sets of webs invariants satisfy certain relations \cite{kup96,CKM} and therefore are not bases. A \emph{web basis} is a subset of the web invariants forming a basis; a web basis is rotation-invariant if the rotation of a basis web with respect to the boundary labels results in another basis web (up to sign). Rotation-invariance is desirable because (among many other reasons) rotation yields an automorphism of the cluster structure on $\C[\Gr(r,n)]$.

For small values of $r$, rotation-invariant web bases are known. The Temperley--Lieb basis for $r=2$ consists of non-crossing matchings between the boundary vertices. Kuperberg \cite{kup96} showed that the subset of $\SL_3$ webs which are \emph{non-elliptic} (meaning that they have no 2-cycles or 4-cycles) form a rotation-invariant basis when $r=3$. Recently, Gaetz, Pechenik, Pfannerer, Striker, and Swanson \cite{GPPSS24} found a rotation-invariant  $\SL_4$ web basis. In particular, their basis disallows 4-cycles which contain an hourglass edge. Finding rotation-invariant $\SL_r$ web bases for $r>4$ is an open problem, although Fraser \cite{fraser23} constructed such a basis for $\C[\Gr(r,n)]_\lambda$ in the special case that $|\lambda|=2r$ (see also \cite{GPPSS-two-column}).

For each of the known web bases above, when $|\lambda|=kr$ there is a bijection between basis webs for $\C[\Gr(r,n)]_\lambda$ (resp. $\W_\lambda(\C^r)$) and semistandard (resp. dual semistandard) Young tableaux of rectangular shape $r\times k$ and content $\lambda$ (see e.g. \cite{KK99, Tym12, GPPSS-promotion}, in addition to the previous references). 

\begin{definition}
Let $\lambda\in\N^n$ with $|\lambda|=kr$, and let $W$ be a (dual) semistandard web of degree $\lambda$. From the boundary label subsets $\mathcal S_W=(S_W(1),\ldots, S_W(n))$, we construct the \emph{tableau of $W$}, denoted $T(W)$, as follows. The shape of $T(W)$ is a rectangle with $r$ rows and $k$ columns. For all $j\in[r]$, the content of the $j$th row of $T(W)$ is the multiset $\{i\in[n]\mid j\in S_W(i)\}$, sorted so that the rows of $T(W)$ are weakly increasing.
\end{definition}

\begin{example} For the $\SL_5$ webs $W$ and $X$ of Figure \ref{fig:SL_5 webs} and Example \ref{ex:boundary-subsets}, the associated $5\times 3$ tableaux are as follows. $T(W)$ is dual semistandard and $T(X)$ is semistandard. Notice the word $w(W)$ can be recovered by recording, in increasing order, the rows containing 1's then the rows containing 2's, and so on. The same is true for $w(X)$ where we record the same row multiple times if it contains multiple copies of the same number. The associated words are often called Yamanouchi words.

\[\ytableausetup{boxsize = 0.45cm, aligntableaux=center} 
T(W)=~
\begin{ytableau}
    1&2&3\\1&4&6\\1&4&7\\2&6&9\\5&8&10
\end{ytableau}~,
\qquad
T(X)=~
\begin{ytableau}
    1&1&1\\2&2&4\\3&5&6\\4&7&8\\6&9&10
\end{ytableau}~.
\]
\end{example}

\begin{theorem}[\cite{KK99, Tym12, GPPSS24}]
\label{thm:rotation-is-promotion}
Let $r\in\{2,3,4\}$, $n=kr$, and $\lambda=(1^n)$.
The correspondence $W\leftrightarrow T(W)$ is a bijection between the set of basis webs $W$ for $\W_\lambda(\C^r)$ and the set of standard Young tableaux in a $k\times r$ rectangle \cite{KK99, Tym12, GPPSS24}. These bijections satisfy several nice properties:
\begin{itemize}
\item Promotion on $T(W)$ corresponds to counterclockwise rotation of $W$ by $2\pi/n$
\cite{petersen2009promotion,GPPSS24}.
\item Evacuation on $T(W)$ corresponds to reflection of $W$ about the line of symmetry passing between boundary vertices $n$ and $1$
\cite{patrias2023tableau,GPPSS24}.
\end{itemize}
\end{theorem}

Theorem~\ref{thm:rotation-is-promotion} implies in particular that the web bases in question are \emph{rotation-invariant}, meaning that the rotation of a basis web is again a basis web. Another special property of these bases is that each basis web $W$ only admits one consistent labeling with boundary condition $\mathcal S_W$:

\begin{theorem}[\cite{GPPSS24,MR4373223}] \label{thm:asw1}
    Let $r\in\{2,3,4\}$ and let $W$ be an $\SL_r$ basis web. Then $a(\mathcal S_W;W)=1.$
\end{theorem}

\subsection{The Fraser--Lam--Le pairing and web duality}

Throughout this section, let $G$ denote a reduced top cell plabic graph of type $(k,n)$, and let $\lambda\in \mathbb{N}^n$ with $|\lambda| = kr$. 

\begin{definition}[\cite{FLL19}]
Each $r$-dimer cover $D$ of $G$ can be viewed as an $\SL_r$ web (called an \emph{$r$-weblike subgraph} of $G$) by ignoring the underlying graph $G$ and only considering the edges in $D$ (with multiplicity). We denote the corresponding web invariant by $\mathbf{D}\in\W_\lambda(\C^r)$.

For $N$ a network on the plabic graph $G$, the \emph{FLL tensor invariant} $\Web_r(N;\lambda)$ is defined as a weighted sum over all $r$-dimer covers $D\in\mathcal{D}_{r,\lambda}(G)$ using edge weights of (\ref{eq:eweight}):
\begin{equation}\label{eq:Webr_def}
\Web_r(N;\lambda)\coloneqq\sum_{D \in \mathcal{D}_{r,\lambda}(G)}\ewt_N(D){\bf D} \quad\in \mathcal{W}_\lambda(\C^r).
\end{equation}
\end{definition}

\begin{definition/theorem} [Immanant map and FLL pairing, \cite{FLL19}] \label{def:FLL_pairing}
Let $\lambda\in \N^n$ with $|\lambda| = kr$. The spaces $\W_\lambda(\C^r)$ and $\C[\Gr(k,n)]_\lambda$ are dual \cite[Thm.\ 4.8]{FLL19}. The \emph{immanant map}
 \[\Imm:(\W_\lambda(\C^r))^*\xrightarrow{\ \ \iso\ \ }\C[\Gr(k,n)]_\lambda\]
is an isomorphism.
Given $\phi\in (\mathcal{W}_\lambda(\C^r))^*$, $\Imm(\phi)$ is defined as the unique $f \in \C[\Gr(k,n)]_\lambda$ which makes the following diagram commute:
\begin{equation}\label{eq:comm_diagram}
    \begin{tikzcd}
    \{\text{networks $N$ on $G$}\} \arrow[rr, "\Web_r(-;\lambda)"] \arrow[d, "\widehat{X}"'] 
    &&\W_\lambda(\C^r) \arrow[d, "\phi"]\\
    \Gr(k,n) \arrow[rr, dashed, "\exists!f =: \Imm(\phi)"]
    &&\C
\end{tikzcd}.
\end{equation}
This duality determines the \emph{FLL pairing} $\langle\cdot,\cdot\rangle:\W_\lambda(\C^r)\otimes \C[\Gr(k,n)]_\lambda\to\C.$ Explicitly, for $\mathbf{W}\in \W_\lambda(\C^r)$ and $f\in\C[\Gr(k,n)]$, we define
\[\langle \mathbf{W}, f\rangle \coloneqq (\Imm^{-1}(f))(\mathbf{W}).\]
\end{definition/theorem}

There is a useful combinatorial characterization of the FLL pairing in terms of counting consistent labelings of a web, which we will use frequently throughout the later sections.

\begin{definition} \label{def:I,S duality}
    Given a list $\mathcal{I} = (I_1,\ldots,I_r)$ of $k$-subsets of $[n]$, and a list $\mathcal{S}=(S(1),\ldots,S(n))$ of subsets of $[r]$, we say that $\mathcal{I}$ and $\mathcal{S}$ are \emph{dual} if $S(i) = \{j\in[r] \mid i \in I_j\}$ for all $i\in[n]$. 
\end{definition}

\begin{proposition}[See Eq.\ 5.16 of \cite{FLL19}]\label{proposition:PairingAsCountingLabelings}
    Let $W$ be an $\SL_r$ web with associated web invariant $\mathbf{W}\in\mathcal{W}_\lambda(\C^r)$. Let $I_1,\ldots,I_r$ be $k$-subsets of $[n]$ with $\plu{I_1}\cdots\plu{I_r}\in\C[\Gr(k,n)]_\lambda$. Let $\mathcal{S}$ be the list of boundary label subsets dual to $(I_1,\ldots,I_r)$.
    Then, 
\begin{equation}\label{eq:PairingAsCountingLabelings}
    \langle {\bf W},\ \plu{I_1}\cdots\plu{I_r}\rangle = a(\mathcal{S};W).
\end{equation}
\end{proposition}

\section{Twists and $r$-dimer covers}\label{sec:TwistsAndDimers}

In this section, we will review some of the previously known expansion formula for cluster variables in various Grassmannian cluster algebras. More specifically, we will state dimer expansion formulas for the image of cluster variables under the twist map as seen in \cite{marsh2016twists, muller2017twist, EMW23}. We then generalize these results to obtain a Laurent expansion for the twist of \emph{any} element of $\mathbb{C}[\Gr(k,n)]$, with respect to an initial cluster of Pl\"ucker coordinates. The coefficients are given in terms of the FLL pairing.

Let $M$ be a full-rank $k \times n$ matrix representing an element of $\text{Gr}(k,n)$ with column vectors $v_1, v_2, \dots v_n \in \mathbb{C}^k$. The \emph{generalized cross-product} $v_1 \times v_2 \times \dots \times v_{k-1}$ is the unique vector $v \in \mathbb{C}^k$ such that $v \cdot w = \det(v_1 \quad v_2 \quad \ldots \quad v_{k-1} \quad w)$ for all $w \in \mathbb{C}^k$.

\begin{definition}\label{definition:twist}
The \emph{twist} of $M$ is the full rank $k \times n$ matrix $\tau(M)$ whose $i$th column is given by $\tau(M)_i= \varepsilon'_i \cdot v_{i+1} \times \dots \times v_{i+k-1}$, where
\[\varepsilon'_i = \begin{cases}
(-1)^{(k-1)(n-i+1)} & i \geq (n-k+2)\\
1 & i \leq (n-k+1)
\end{cases}\]
taking indices modulo $n$ as appropriate. Explicitly,
$$\tau(M)_i = \begin{cases} 
(-1)^{k-n+i-1} v_1 \times v_2 \times \dots \times v_{i-n+k-1} \times v_{i+1} \times \dots \times v_n & \mathrm{~if~} i\geq n-k+2 \\
v_{i+1} \times v_{i+2} \times \dots \times v_{i+k-1} & \mathrm{~if~} i\leq n-k+1 \end{cases}.$$
\end{definition}

\begin{remark}\label{remark:twistdefinitionsubtly}
Definition \ref{definition:twist} uses the conventions of \cite{EMW23}. This definition is a ``right" adaptation of Marsh--Scott's left twist as in \cite{marsh2016twists}, but is not the same as the right twist as in \cite{muller2017twist}. See Section 2.5 of \cite{EMW23} for further discussion. 
\end{remark}

Recall the definition of the face labels $I_{\face}$ of a plabic graph $G$ from Section~\ref{subsec:gr_clusteralgs}. The \emph{face weight} of an $r$-dimer cover $D$ of a plabic graph $G$ is:
$$ \fwt_G(D)\coloneqq\prod_{\face\in F(G)}\plu{I_\face}^{rW_\face-D_\face-r},$$
where $W_{\face}$ the number of white vertices bordering face $\face$, and $D_\face$ the number of non-boundary edges of $\face$ used in $D$.

The following theorem is due to Marsh--Scott, adapted to our conventions as in \cite{EMW23}.
\begin{theorem}\cite[Thm.~1.1] {marsh2016twists},
\cite[Thm.~3.4]{EMW23}
\label{Marsh-ScottTwistThm}
Let $G$ be a reduced top cell plabic graph of type $(k,n)$ and let $I \in {[n] \choose k}$. The twist of $\Delta_I$ is given by
\[\twist(\Delta_I)= \sum_{D\in\mathcal{D}_{1,\delta^I}(G)}\fwt_G(D).\]
\end{theorem}

This brings us to our first main result which generalizes the above.

\begin{theorem}\label{thm:TwistGivenBySumOfPairings}
Let $G$ be a reduced top cell plabic graph of type $(k,n)$. Let $f \in \C[\Gr(k,n)]_\lambda$ with $|\lambda|=kr$. Using the FLL pairing, the twist of $f$ is given by
\begin{equation} \label{eq:TwistPairing}
\twist(f)=\sum_{D \in \mathcal{D}_{r,\lambda}(G)}\langle \mathbf{D}, f\rangle \fwt_G(D).    
\end{equation}
\end{theorem}

\begin{proof}
Since $\twist$ is a ring endomorphism,
and the FLL pairing is bilinear, both sides of \eqref{eq:TwistPairing} are linear in $f$. So it suffices to prove the claim when $f$ is a product of Pl\"{u}cker coordinates $\Delta_{I_1} \cdots \Delta_{I_r}$, as these span $\C[\Gr(k,n)]$.

For the base case $r = 1$, we have $f=\Delta_I$ and $\lambda=\delta^I$. The single-element list $(I)$ is dual to $\S=(S(1),\ldots,S(n)),$ where $S(i)=\{1\}$ if $i\in I$ and $S(i)=\emptyset$ otherwise.
For each 1-dimer cover $D\in\mathcal{D}_{1,\delta^I}(G),$ there is a unique consistent labeling $\L\in A(\S;D)$, namely $L(\mathbf{e})=\{1\}$ for every edge $\mathbf{e}$ of $D$. By equation (\ref{eq:PairingAsCountingLabelings}) we have $\langle \mathbf{D},\Delta_I\rangle=a(\mathcal{S}; D)=1,$ thus the base case reduces to Theorem \ref{Marsh-ScottTwistThm}.

For the inductive step, let $r>1$ and $f = \Delta_{I_1} \cdots \Delta_{I_r} \in \C[\Gr(k,n)]_{\lambda}$.
For any choice of $r'\in[r-1]$, let $f'=\Delta_{I_1}\cdots\Delta_{I_{r'}}\in\C[\Gr(k,n)]_{\lambda'}$ and $f''=\Delta_{I_{r'+1}}\cdots\Delta_{I_r}\in\C[\Gr(k,n)]_{\lambda-\lambda'}$, so that $f=f'f''$.

Let $\mathcal{S}=(S(1),\ldots, S(n))$ be the list of boundary label subsets dual to $(I_1,\ldots, I_r)$ (Definition \ref{def:I,S duality}). Similarly, let $\mathcal{S}'$ and $\mathcal{S}''$ be dual to $ (I_1,\ldots,I_{r'})$ and $(I_{r'+1},\ldots,I_r)$, respectively. It is clear that $S(i)=S'(i)\sqcup S''(i)$ for all $i\in[n]$.

By induction on $r$, and since $\twist$ is a ring endomorphism, we have:
\begin{align}
    \twist(f)=\twist(f'f'') &= \twist(f')\twist(f'') \\
    &= \label{eq:inductive-step} \left(\sum_{D' \in \mathcal{D}_{r',\lambda'}(G)} \langle \mathbf{D'}, f' \rangle \fwt_G(D') \right) \left( \sum_{D'' \in \mathcal{D}_{r-r',\lambda-\lambda'}(G)} \langle \mathbf{D''}, f'' \rangle \fwt_G(D'') \right) \\
    &= \label{eq:pairing-to-a} \left(\sum_{D' \in \mathcal{D}_{r',\lambda'}(G)} a(\S; D') \fwt_G(D') \right) \left( \sum_{D'' \in \mathcal{D}_{r-r',\lambda-\lambda'}(G)} a(\S'';D'') \fwt_G(D'') \right) \\
    &= \label{eq:product-of-a} \sum_{D \in \mathcal{D}_{r,\lambda}(G)} a(\S; D) \fwt_G(D) \\
    &= \label{eq:final} \sum_{D \in \mathcal{D}_{r,\lambda}(G)} \langle \mathbf{D}, f\rangle \fwt_G(D).
\end{align}
 Here \eqref{eq:inductive-step} is by the inductive hypothesis, and \eqref{eq:pairing-to-a} and \eqref{eq:final} follow from \cite[Eq.~5.16]{FLL19}.

We now justify (\ref{eq:product-of-a}).
For each $D\in\mathcal{D}_{r,\lambda}(G)$,  the map 
\[
\bigsqcup_{\substack{D' \in \mathcal{D}_{r',\lambda'}(G) \\ D'' \in \mathcal{D}_{r-r',\lambda-\lambda'}(G) \\ D' \sqcup D'' = D}} A(\S'; D') \times A(\S''; D'') \to A(\S; D),
\]
defined by $L(\mathbf{e})=L'(\mathbf{e}) \sqcup L''(\mathbf{e})$ for each edge $\mathbf{e}$ of $D$, is easily seen to be a bijection. Indeed, given a consistent $r$-labeling $\mathcal{L}\in A(\mathcal{S};D)$ we can recover $D',D''$ and consistent labelings thereof by restricting to the label sets $\{1,\ldots,r'\}$ and $\{r'+1,\ldots,r\}$, respectively. Finally, it is immediate from the definition that $\fwt_G$ is multiplicative with respect to this map: if $D' \sqcup D''=D$, then $\fwt_G(D')\fwt_G(D'')=\fwt_G(D)$.
\end{proof}

The following equation \cite[Eq.\ 5.17]{FLL19} is directly implied from the definition of the FLL pairing \eqref{eq:comm_diagram}:
\begin{equation}\label{eq:PairingWithWebr}
    \langle \Web_r(N;\lambda),f\rangle = f(\widehat{X}(N)).
\end{equation}
As a corollary to Theorem \ref{thm:TwistGivenBySumOfPairings}, we obtain a twisted version of equation \eqref{eq:PairingWithWebr}. We first introduce some notation.

\begin{definition}
Let $G$ be a reduced top cell plabic graph of type $(k,n)$, $N$ a network on $G$, and $r\ge 1$.
We define $\fwt_N(D)$, the \emph{face weight of $D$ with respect to $N$}, to be the function $\fwt_G(D)$ evaluated at the point $\widehat{X}(N)\in\Gr(k,n)$:
\[\fwt_N(D)\coloneqq(\fwt_G(D))(\widehat{X}(N))\quad\in\C.\]
In analogy to the FLL tensor invariant $\Web_r(N;\lambda)$ from \eqref{eq:Webr_def}, we define the \emph{twisted FLL tensor invariant} $\Web_r^\twist(N;\lambda)$, which replaces the edge weight by the face weight:
\begin{equation}\label{eq:TwistedFLLTensor}
\Web_r^\twist(N;\lambda)\coloneqq\sum_{D \in \mathcal{D}_{r,\lambda}(G)}\fwt_N(D){\bf D} \quad\in\W_\lambda(\C^r).
\end{equation}
\end{definition}

\begin{corollary}\label{cor:twistedtensorpairing} For $N$ a $(k,n)$-network and $f\in\C[\Gr(k,n)]_\lambda,$ we have
    \begin{equation}\label{eq:PairingWithTwistedWebr}
    \langle \text{Web}_r^\twist(N;\lambda),f\rangle
    =(\twist(f))(\widehat{X}(N)).        
    \end{equation}
\end{corollary}

\begin{proof}
In Theorem \ref{thm:TwistGivenBySumOfPairings}, we showed $\twist(f) =\sum_{D \in \mathcal{D}_{r,\lambda}(G)}\langle \mathbf{D}, f\rangle \fwt_G(D)$. The proof follows by evaluating both sides of this equation at the point $\widehat{X}(N) \in \Gr(k,n)$:
\begin{align*}
    (\twist(f))(\widehat X(N))
    &=\sum_{D \in \mathcal{D}_{r,\lambda}(G)}\langle \mathbf{D}, f\rangle \fwt_N(D)
    &\text{(definition of $\fwt_N(D)$)}\\
    &=\left\langle\left(\sum_{D \in \mathcal{D}_{r,\lambda}(G)} \fwt_N(D)\mathbf{D}\right) ,\ f\right\rangle\\
    &=\langle \text{Web}_r^\twist(N;\lambda),f\rangle.
    &\text{(definition of $\Web_r^\tau(N;\lambda)$)}
\end{align*}
\end{proof}

\section{Computing web duals}\label{sec:ComputingDuals}

Our objective in this section is to compute the FLL pairing (Definition \ref{def:FLL_pairing}) on basis webs in order to make the formula given in Theorem \ref{thm:TwistGivenBySumOfPairings} more explicit. Note that when $\lambda = (1^n)$, we have that $\W_\lambda(\C^r)= \C[\Gr(r,n)]_\lambda$, so in this setting the FLL pairing is defined as $\langle\cdot,\cdot\rangle:\C[\Gr(r,n)]_\lambda\otimes \C[\Gr(k,n)]_\lambda\to\C$.

  Fraser--Lam--Le \cite{FLL19} observed that when $(r,k)=(2,3),\ (3,2),$ and $(3,3)$ and $\lambda = (1^n)$ for $n = kr$, the known web bases are dual under this pairing, and duality corresponds to transposing the standard Young tableaux associated to basis webs. We want to note that the observations made in their Appendix were jointly written with Addabbo--Bucher--Clearman--Escobar--Ma--Oh--Vogel. After this, Elkin--Musiker--Wright \cite{EMW23} showed the same duality between basis webs for $\W_\lambda(\C^3)$ and $\C[\Gr(3,8)]_\lambda$, when $\lambda=(2,1,\ldots,1)\in \N^8$. 
  We will show here that this correspondence also holds when $\lambda=(1^{12})$ and $(r,k)=(4,3)$ (or $(3,4)$).

Let $\lambda=(1^n)$ and let $\B$ denote Kuperberg's non-elliptic web basis for the 462-dimensional $\C$-vector space $\C[\Gr(3,12)]_\lambda$. Both $\mathcal{B}$ and the dual basis $\B^*$ for the space $\C[\Gr(4,12)]_\lambda\iso(\C[\Gr(3,12)]_\lambda)^*$ are depicted in Tables \ref{tablep1} and \ref{tablep2}. Each cell in the table depicts a basis web $X_i\in \B$ (right) and its dual $W_i=X_i^*\in\B^*$ (left), as well as their corresponding words $w(W_i)$ and $w(X_i)$. The $462$ basis webs in $\B$ can be grouped into $32$ orbits up to rotation and reflection. We depict one representative $W_i\in\B^*$ for each of these orbits. We order the table in lexicographic order on the words $w(W_i)$, choosing representative $W_i$ whose word is lexicographically minimal in each orbit.

For the remainder of this paper, for $1\le i\le 32$, we will let $W_i\in \B^*$ and $X_i\in \B$ denote the webs in cell $i$ of the table. Our main result in this section is the following. In the following, let ``$t$'' denote transpose.

\begin{theorem}\label{thm:DualityComputation}
   The bases $\B$ and $\B^*$, as depicted in Tables \ref{tablep1} and \ref{tablep2}, are dual with respect to the FLL pairing, up to sign. Precisely, 
   \[\langle \mathbf{W}_i, \mathbf{X}_j \rangle = \begin{cases}
       \sign(X_j)&i=j,\\
       0 &i\ne j.
   \end{cases}\]
   Moreover, the corresponding tableaux satisfy $T(W_i)=T(X_i)^t$. 
\end{theorem}

\begin{remark}\label{rmk:SymmetricFLL}
When $\lambda = (1^n)$, the immanant map is an isomorphism between $\mathcal{W}_\lambda(\C^r)^*$ and $   \mathcal{W}_{\lambda}(\C^k)$. In \cite[Theorem 8.1]{FLL19}, the authors show that this isomorphism is given by $\mathcal{W}_\lambda(\C^r)^* \to   \mathcal{W}_{\lambda}(\C^k) \otimes \epsilon$ as representations of $S_n$ where $\epsilon$ is the sign representation. From this description, we deduce that for $\mathbf{W} \in \mathcal{W}_\lambda(\C^r)$ and $\mathbf{X} \in \mathcal{W}_{\lambda}(\C^k)$, the FLL-pairing is symmetric i.e. $\langle \mathbf{W},\mathbf{X} \rangle = \langle \mathbf{X},\mathbf{W} \rangle$.
\end{remark}

We prove Theorem \ref{thm:DualityComputation} by computing many pairings $\langle \mathbf{W},\mathbf{X}\rangle$. This is \emph{a priori} a very large number of calculations (given that $\B$ and $\B^*$ each consist of $462$ basis webs). However, through a set of lemmas (i.e. Lemmas \ref{lem:RotationBothWebs}, \ref{lem: disjoint forks}, \ref{lem:DisconnectedConveniently}, \ref{lem:pairing0general}, and \ref{lem:pairing0r=4k=3}), which utilize rotational symmetry, $\SL_3$ skein relations, identify pairings that vanish based on local features, and Proposition \ref{proposition:PairingAsCountingLabelings}, we are able to dramatically reduce the number of computations.

 Recall we let $\rho$ denote a counterclockwise $2\pi/n$ rotation of $\mathcal{W}_\lambda(\C^r)$ web where $\lambda \in \mathbb{N}^n$ and, by abuse of notation, also the induced function on web invariants of $\mathcal{W}_\lambda(\C^r)$. 

\begin{lemma}\label{lem:RotationBothWebs}
 Let $\lambda \in \mathbb{N}^n$ be such that  $\vert \lambda \vert = kr$. If $\mathbf{W} \in \mathcal{W}_\lambda(\mathbb{C}^r)$ and $\mathbf{X} \in \mathbb{C}[\Gr(k,n)]$ are two web invariants, then  
 \[
\sign(\rho(X))\langle \mathbf{W},\mathbf{X} \rangle = \sign(X)\langle \rho(\mathbf{W}),\rho(\mathbf{X})\rangle.
\]
\end{lemma}

\begin{proof}
Recall that one way to compute $\langle \mathbf{W}, \mathbf{X}\rangle$ is to write $\mathbf{X}$ as a sum of products of Pl\"ucker coordinates. We switch to the FP sign conventions here for their consistent behavior under wrench relations (see Remark \ref{rem:props of FP signs}):\[
\langle \mathbf{W}, \mathbf{X}\rangle = \sign(X) \langle \mathbf{W}, \mathbf{X}^{FP} \rangle = \sign(X) \sum_\ell c_\ell \langle \mathbf{W},\mathbf{X}_\ell^{FP} \rangle = \sign(X) \sum_\ell c_\ell a(\mathcal{S}_\ell,W)
\]
where $S_\ell$ is the boundary label condition determined by $X_\ell$, a product of Pl\"ucker coordinates. Now, we do the same after rotating each web, \begin{align*}
\langle \rho(\mathbf{W}), \rho(\mathbf{X}\rangle) &= \sign(\rho(X)) \langle \rho(\mathbf{W}), \rho(\mathbf{X})^{FP} \rangle\\
&= \sign(\rho(X)) \sum_\ell c_\ell \langle \rho(\mathbf{W}),\rho(\mathbf{X}_\ell)^{FP} \rangle\\
&= \sign(\rho(X)) \sum_\ell c_\ell a(\rho(\mathcal{S}_\ell),\rho(W),)\\
&= \sign(\rho(X)) \sum_\ell c_\ell a(\mathcal{S}_\ell,W),
\end{align*}
where we recall $\rho(\mathcal{S})$ denotes a cyclic permutation.
\end{proof}

The required computations can be further reduced using the next series of lemmas. These concern local configurations of two webs which guarantee their web invariants pair to 0. In our proof of Theorem \ref{thm:DualityComputation}, we will sometimes apply these lemmas after using skein relations which introduce crossings into our diagrams. Therefore, many lemmas are phrased in terms of tensor diagrams. We begin with a basic but important structure in a web.

\begin{definition} For $i<j\in[n]$, we say a tensor diagram $W$ has a \emph{fork at $(i,j)$} if boundary vertices $i$ and $j$ are each connected to a common internal vertex by a single edge.
\end{definition}

A given fork can be preserved when expanding an $\SL_3$ tensor invariant from a tensor diagram in terms of Pl{\"u}cker coordinates. 

\begin{lemma}\label{lemma: plu expansion fork}
    Let $X$ be an $\SL_3$ tensor diagram, and suppose $X$ has a fork between boundary vertices $i$ and $j$. Then there exists an expansion of $\mathbf{X}$ as a linear combination of products of Pl\"ucker coordinates, $\mathbf{X}=\sum_\ell\pm \mathbf{X}_\ell,$ such that each term $\mathbf{X}_\ell$ has some factor  $\plu{I_\ell}$ with $i,j\in I_\ell$.
\end{lemma}

\begin{proof}
     Let $k$ be the Pl\"{u}cker degree of $X$. We will induct on $w$, the number of internal white vertices of $X$.

    For the base case, if $w=k,$ then $X$ has $k$ internal white vertices and no internal black vertices, so $\mathbf{X}$ is the product of $k$ Pl\"ucker coordinates. By the assumption that $i$ and $j$ are connected by a fork,  one of the Pl\"ucker factors $\plu I$ must have $i,j\in I.$

    For the inductive step, suppose $w>k.$ Consider the connected component of the graph containing vertices $i$ and $j$, which we denote by $X'$. If $X'$ is a tripod (that is, $\mathbf{X}'=\Delta_I$ with $i,j\in I$), then taking any Pl\"ucker expansion of $\frac{\mathbf{X}}{\mathbf{X'}}$ yields a desired expansion of $\mathbf{X}$. If $X'$ is not a tripod, then the fork $(i,j)$ in $X$ is contained in a configuration as below.
    
\begin{center}
\begingroup%
  \makeatletter%
  \providecommand\color[2][]{%
    \errmessage{(Inkscape) Color is used for the text in Inkscape, but the package 'color.sty' is not loaded}%
    \renewcommand\color[2][]{}%
  }%
  \providecommand\transparent[1]{%
    \errmessage{(Inkscape) Transparency is used (non-zero) for the text in Inkscape, but the package 'transparent.sty' is not loaded}%
    \renewcommand\transparent[1]{}%
  }%
  \providecommand\rotatebox[2]{#2}%
  \newcommand*\fsize{\dimexpr\f@size pt\relax}%
  \newcommand*\lineheight[1]{\fontsize{\fsize}{#1\fsize}\selectfont}%
  \ifx\svgwidth\undefined%
    \setlength{\unitlength}{72.13164562bp}%
    \ifx\svgscale\undefined%
      \relax%
    \else%
      \setlength{\unitlength}{\unitlength * \real{\svgscale}}%
    \fi%
  \else%
    \setlength{\unitlength}{\svgwidth}%
  \fi%
  \global\let\svgwidth\undefined%
  \global\let\svgscale\undefined%
  \makeatother%
  \begin{picture}(1,1.02642935)%
    \lineheight{1}%
    \setlength\tabcolsep{0pt}%
    \put(0,0){\includegraphics[width=\unitlength,page=1]{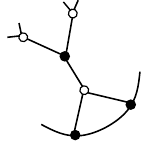}}%
    \put(0.4065886,0.01459273){\color[rgb]{0,0,0}\rotatebox{-0.15281934}{\makebox(0,0)[lt]{\lineheight{0}\smash{\begin{tabular}[t]{l}$j$\end{tabular}}}}}%
    \put(0.91060152,0.26808761){\color[rgb]{0,0,0}\rotatebox{-0.15281934}{\makebox(0,0)[lt]{\lineheight{0}\smash{\begin{tabular}[t]{l}$i$\end{tabular}}}}}%
  \end{picture}%
\endgroup%

\end{center}

Applying a wrench relation at the highlighted edges, we obtain the following difference, where the signs come from the FP conventions.

\begin{center}
\begin{tikzpicture}

\node(0) at (0,0) {
\begingroup%
  \makeatletter%
  \providecommand\color[2][]{%
    \errmessage{(Inkscape) Color is used for the text in Inkscape, but the package 'color.sty' is not loaded}%
    \renewcommand\color[2][]{}%
  }%
  \providecommand\transparent[1]{%
    \errmessage{(Inkscape) Transparency is used (non-zero) for the text in Inkscape, but the package 'transparent.sty' is not loaded}%
    \renewcommand\transparent[1]{}%
  }%
  \providecommand\rotatebox[2]{#2}%
  \newcommand*\fsize{\dimexpr\f@size pt\relax}%
  \newcommand*\lineheight[1]{\fontsize{\fsize}{#1\fsize}\selectfont}%
  \ifx\svgwidth\undefined%
    \setlength{\unitlength}{72.13164562bp}%
    \ifx\svgscale\undefined%
      \relax%
    \else%
      \setlength{\unitlength}{\unitlength * \real{\svgscale}}%
    \fi%
  \else%
    \setlength{\unitlength}{\svgwidth}%
  \fi%
  \global\let\svgwidth\undefined%
  \global\let\svgscale\undefined%
  \makeatother%
  \begin{picture}(1,1.02642935)%
    \lineheight{1}%
    \setlength\tabcolsep{0pt}%
    \put(0,0){\includegraphics[width=\unitlength,page=1]{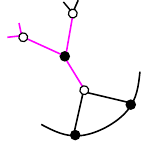}}%
    \put(0.4065886,0.01459273){\color[rgb]{0,0,0}\rotatebox{-0.15281934}{\makebox(0,0)[lt]{\lineheight{0}\smash{\begin{tabular}[t]{l}$j$\end{tabular}}}}}%
    \put(0.91060152,0.26808761){\color[rgb]{0,0,0}\rotatebox{-0.15281934}{\makebox(0,0)[lt]{\lineheight{0}\smash{\begin{tabular}[t]{l}$i$\end{tabular}}}}}%
  \end{picture}%
\endgroup%
};
\node(1) at (0,-1.75) {$X$};
\node(2) at (2,0) {$=$};
\node(3) at (4,0) {
\begingroup%
  \makeatletter%
  \providecommand\color[2][]{%
    \errmessage{(Inkscape) Color is used for the text in Inkscape, but the package 'color.sty' is not loaded}%
    \renewcommand\color[2][]{}%
  }%
  \providecommand\transparent[1]{%
    \errmessage{(Inkscape) Transparency is used (non-zero) for the text in Inkscape, but the package 'transparent.sty' is not loaded}%
    \renewcommand\transparent[1]{}%
  }%
  \providecommand\rotatebox[2]{#2}%
  \newcommand*\fsize{\dimexpr\f@size pt\relax}%
  \newcommand*\lineheight[1]{\fontsize{\fsize}{#1\fsize}\selectfont}%
  \ifx\svgwidth\undefined%
    \setlength{\unitlength}{72.13164562bp}%
    \ifx\svgscale\undefined%
      \relax%
    \else%
      \setlength{\unitlength}{\unitlength * \real{\svgscale}}%
    \fi%
  \else%
    \setlength{\unitlength}{\svgwidth}%
  \fi%
  \global\let\svgwidth\undefined%
  \global\let\svgscale\undefined%
  \makeatother%
  \begin{picture}(1,1.02642935)%
    \lineheight{1}%
    \setlength\tabcolsep{0pt}%
    \put(0,0){\includegraphics[width=\unitlength,page=1]{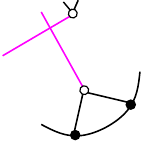}}%
    \put(0.4065886,0.01459273){\color[rgb]{0,0,0}\rotatebox{-0.15281934}{\makebox(0,0)[lt]{\lineheight{0}\smash{\begin{tabular}[t]{l}$j$\end{tabular}}}}}%
    \put(0.91060152,0.26808761){\color[rgb]{0,0,0}\rotatebox{-0.15281934}{\makebox(0,0)[lt]{\lineheight{0}\smash{\begin{tabular}[t]{l}$i$\end{tabular}}}}}%
  \end{picture}%
\endgroup%
};
\node(4) at (6,0) {$-$};
\node(5) at (8,0) {
\begingroup%
  \makeatletter%
  \providecommand\color[2][]{%
    \errmessage{(Inkscape) Color is used for the text in Inkscape, but the package 'color.sty' is not loaded}%
    \renewcommand\color[2][]{}%
  }%
  \providecommand\transparent[1]{%
    \errmessage{(Inkscape) Transparency is used (non-zero) for the text in Inkscape, but the package 'transparent.sty' is not loaded}%
    \renewcommand\transparent[1]{}%
  }%
  \providecommand\rotatebox[2]{#2}%
  \newcommand*\fsize{\dimexpr\f@size pt\relax}%
  \newcommand*\lineheight[1]{\fontsize{\fsize}{#1\fsize}\selectfont}%
  \ifx\svgwidth\undefined%
    \setlength{\unitlength}{73.18484089bp}%
    \ifx\svgscale\undefined%
      \relax%
    \else%
      \setlength{\unitlength}{\unitlength * \real{\svgscale}}%
    \fi%
  \else%
    \setlength{\unitlength}{\svgwidth}%
  \fi%
  \global\let\svgwidth\undefined%
  \global\let\svgscale\undefined%
  \makeatother%
  \begin{picture}(1,1.01165812)%
    \lineheight{1}%
    \setlength\tabcolsep{0pt}%
    \put(0,0){\includegraphics[width=\unitlength,page=1]{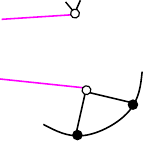}}%
    \put(0.41512833,0.01438273){\color[rgb]{0,0,0}\rotatebox{-0.15281934}{\makebox(0,0)[lt]{\lineheight{0}\smash{\begin{tabular}[t]{l}$j$\end{tabular}}}}}%
    \put(0.91188804,0.26422959){\color[rgb]{0,0,0}\rotatebox{-0.15281934}{\makebox(0,0)[lt]{\lineheight{0}\smash{\begin{tabular}[t]{l}$i$\end{tabular}}}}}%
  \end{picture}%
\endgroup%
};
\node(6) at (4,-1.75){$X_1$};
\node(7) at (8,-1.75){$X_2$};
\end{tikzpicture}
\end{center}

    Both $X_1$ and $X_2$ have forks connecting $i$ and $j$ and have exactly $w-1$ internal white vertices. Thus by the inductive hypothesis, $\mathbf{X}_1$ and $\mathbf{X}_2$ each have a desired Pl\"ucker expansion. Their difference is a desired expansion of $\mathbf{X}$.
\end{proof}

We can now immediately determine many pairs of webs with trivial FLL-pairings. 

\begin{lemma}\label{lem: disjoint forks}
  Let $\lambda \in \mathbb{N}^n$ with $|\lambda|=3r$. Let $\mathbf{W}\in\mathcal{W}_\lambda(\C^r)$ and $\mathbf{X}\in\C[\Gr(3,n)]_\lambda$ be invariants corresponding to tensor diagrams $W$ and $X$. If both $W$ and $X$ have forks at $(i,j)$, then $\langle \mathbf{W}, \mathbf{X}\rangle= 0.$
\end{lemma}

\begin{proof}
If $i,j$ are connected by a fork in $X$, then by Lemma \ref{lemma: plu expansion fork} there is always a way to expand $\mathbf{X}$ as a sum of products of Pl\"ucker coordinates, $\mathbf{X}=\sum_\ell \pm \mathbf{X}_\ell,$ such that each term $\mathbf{X}_\ell$ has some factor  $\plu{I_\ell}$ with $i,j\in I_\ell$. 

Each $X_\ell$ prescribes a coloring of the boundary edges of $W$ where $i,j$ have the same color, but if $i,j$ are connected by a fork in $W$ this coloring is not proper. Thus $\langle \mathbf{W}, \mathbf{X}_\ell\rangle=0$ for all $\ell$, and by bilinearity $\langle \mathbf{W}, \mathbf{X}\rangle=\sum_\ell \pm \langle \mathbf{W}, \mathbf{X}_\ell\rangle=0.$ 
\end{proof}

Lemma \ref{lem: disjoint forks} applies to many pairs of webs. For instance, in Tables \ref{tablep1} and \ref{tablep2}, we see that for all $s \leq 20$, $X_s$ has a size 3 fork, i.e., three consecutive boundary vertices which are incident to a common interior vertex. Meanwhile, for each $W_t$ with $t > 20$, and each set of three consecutive boundary vertices, at least two are in a fork. Therefore, if $t > 20$ and $s \leq 20$, for any $\ell$ and $m$ we immediately see $\langle \rho^\ell(\mathbf{W}_j),\rho^m(\mathbf{X}_i)\rangle = 0$.

Interpreting the FLL pairing in terms of coloring gives us more configurations which must pair to 0.

\begin{lemma}\label{lem:DisconnectedConveniently}
Let $\lambda \in \mathbb{N}^n$ with $\vert \lambda \vert = kr$.  
Let $\mathbf{W} \in \mathcal{W}_\lambda(\mathbb{C}^r)$ be an invariant associated to a tensor diagram $W$ which contains a boundary configuration as below. We emphasize that the depicted internal black vertex is not adjacent to any other vertices. Let $A_1, A_2,\ldots,A_a$ be the subsets of $[n]$ corresponding to the sets of boundary vertices in this connected component where all distinct pairs in each $A_i$ are in a fork and $A_i \cap A_j = \emptyset$ for all $i \neq j$. Let $\mathbf{X} \in \mathbb{C}[\Gr(k,n)]_\lambda$ be an invariant associated to a tensor diagram which has a factor $\Delta_I$ where $I \cap A_i \neq \emptyset$ for each set $A_i, 1 \leq i \leq a$. Then, $\langle \mathbf{W},\mathbf{X}\rangle =  0$. 

\begin{center}
\begingroup%
  \makeatletter%
  \providecommand\color[2][]{%
    \errmessage{(Inkscape) Color is used for the text in Inkscape, but the package 'color.sty' is not loaded}%
    \renewcommand\color[2][]{}%
  }%
  \providecommand\transparent[1]{%
    \errmessage{(Inkscape) Transparency is used (non-zero) for the text in Inkscape, but the package 'transparent.sty' is not loaded}%
    \renewcommand\transparent[1]{}%
  }%
  \providecommand\rotatebox[2]{#2}%
  \newcommand*\fsize{\dimexpr\f@size pt\relax}%
  \newcommand*\lineheight[1]{\fontsize{\fsize}{#1\fsize}\selectfont}%
  \ifx\svgwidth\undefined%
    \setlength{\unitlength}{147.25963923bp}%
    \ifx\svgscale\undefined%
      \relax%
    \else%
      \setlength{\unitlength}{\unitlength * \real{\svgscale}}%
    \fi%
  \else%
    \setlength{\unitlength}{\svgwidth}%
  \fi%
  \global\let\svgwidth\undefined%
  \global\let\svgscale\undefined%
  \makeatother%
  \begin{picture}(1,0.64106513)%
    \lineheight{1}%
    \setlength\tabcolsep{0pt}%
    \put(0,0){\includegraphics[width=\unitlength,page=1]{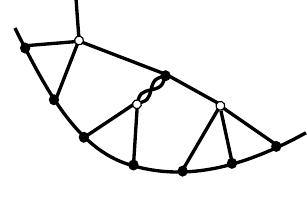}}%
    \put(-0.00120085,0.33580765){\color[rgb]{0,0,0}\rotatebox{-0.15281934}{\makebox(0,0)[lt]{\lineheight{0}\smash{\begin{tabular}[t]{l}$A_1$\end{tabular}}}}}%
    \put(0.25379387,0.04936907){\color[rgb]{0,0,0}\rotatebox{-0.15281934}{\makebox(0,0)[lt]{\lineheight{0}\smash{\begin{tabular}[t]{l}$A_2$\end{tabular}}}}}%
    \put(0.735319,0.00685128){\color[rgb]{0,0,0}\rotatebox{-0.15281934}{\makebox(0,0)[lt]{\lineheight{0}\smash{\begin{tabular}[t]{l}$A_3$\end{tabular}}}}}%
  \end{picture}%
\endgroup%

\end{center}
\end{lemma}
\begin{proof}

Let $\mathbf{X} = \Delta_I \mathbf{X}'$ and arbitrary write $\mathbf{X}'$ as a sum of products of Pl{\"u}cker coordinates, $\mathbf{X} = \Delta_i \sum_\ell \pm \mathbf{X}'_\ell$. Any term $\Delta_i\mathbf{X}'_\ell$ gives an impossible coloring condition on $W$. This is because there would be no way to color an edge incident to the designated black vertex with the color assigned to $\Delta_I$. Therefore, $\langle \mathbf{W},\mathbf{X} \rangle = 0$.
\end{proof}

Combining the previous two results gives us a pairing which can be used on many pairs of web invariants for which Lemma \ref{lem: disjoint forks} does not apply. 

\begin{lemma}\label{lem:pairing0general} 
Let $\lambda = (1^{3r})$. Let $1\le b<a\le r$. Given web invariants $\mathbf{W}\in\C[\Gr(r,3r)]_\lambda$ and $\mathbf{X}\in\C[\Gr(3,3r)]_\lambda$, suppose there exists a cyclic interval of $a+2$ boundary vertices such that the webs $W$ and $X$ have the boundary conditions shown below. Then $\langle \mathbf{W},\mathbf{X}\rangle= \langle \mathbf{X},\mathbf{W}\rangle = 0$. 
    \begin{align*}
    W&=\adjustbox{valign = c}{\ \scalebox{1.4}{
\begingroup%
  \makeatletter%
  \providecommand\color[2][]{%
    \errmessage{(Inkscape) Color is used for the text in Inkscape, but the package 'color.sty' is not loaded}%
    \renewcommand\color[2][]{}%
  }%
  \providecommand\transparent[1]{%
    \errmessage{(Inkscape) Transparency is used (non-zero) for the text in Inkscape, but the package 'transparent.sty' is not loaded}%
    \renewcommand\transparent[1]{}%
  }%
  \providecommand\rotatebox[2]{#2}%
  \newcommand*\fsize{\dimexpr\f@size pt\relax}%
  \newcommand*\lineheight[1]{\fontsize{\fsize}{#1\fsize}\selectfont}%
  \ifx\svgwidth\undefined%
    \setlength{\unitlength}{157.68441028bp}%
    \ifx\svgscale\undefined%
      \relax%
    \else%
      \setlength{\unitlength}{\unitlength * \real{\svgscale}}%
    \fi%
  \else%
    \setlength{\unitlength}{\svgwidth}%
  \fi%
  \global\let\svgwidth\undefined%
  \global\let\svgscale\undefined%
  \makeatother%
  \begin{picture}(1,0.40026669)%
    \lineheight{1}%
    \setlength\tabcolsep{0pt}%
    \put(0,0){\includegraphics[width=\unitlength,page=1]{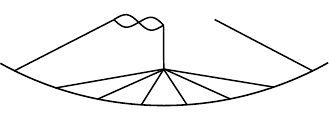}}%
    \put(0.48651777,0.19663894){\color[rgb]{0,0,0}\rotatebox{90.100147}{\makebox(0,0)[lt]{\lineheight{0}\smash{\begin{tabular}[t]{l}\tiny{$r-a$}\end{tabular}}}}}%
    \put(0.4151143,0.35419504){\color[rgb]{0,0,0}\rotatebox{-6.2034378}{\makebox(0,0)[lt]{\lineheight{0}\smash{\begin{tabular}[t]{l}\tiny{$b$}\end{tabular}}}}}%
    \put(0.52982195,0.34519807){\color[rgb]{0,0,0}\rotatebox{6.6554077}{\makebox(0,0)[lt]{\lineheight{0}\smash{\begin{tabular}[t]{l}\tiny{$a-b$}\end{tabular}}}}}%
    \put(0.48930095,0.00337571){\color[rgb]{0,0,0}\rotatebox{-0.15281934}{\makebox(0,0)[lt]{\lineheight{0}\smash{\begin{tabular}[t]{l}\tiny{$a$}\end{tabular}}}}}%
    \put(0,0){\includegraphics[width=\unitlength,page=2]{generallemma.1.pdf}}%
  \end{picture}%
\endgroup%
}}\\ 
    X&=\adjustbox{valign=c}{\ \scalebox{1.4}{
\begingroup%
  \makeatletter%
  \providecommand\color[2][]{%
    \errmessage{(Inkscape) Color is used for the text in Inkscape, but the package 'color.sty' is not loaded}%
    \renewcommand\color[2][]{}%
  }%
  \providecommand\transparent[1]{%
    \errmessage{(Inkscape) Transparency is used (non-zero) for the text in Inkscape, but the package 'transparent.sty' is not loaded}%
    \renewcommand\transparent[1]{}%
  }%
  \providecommand\rotatebox[2]{#2}%
  \newcommand*\fsize{\dimexpr\f@size pt\relax}%
  \newcommand*\lineheight[1]{\fontsize{\fsize}{#1\fsize}\selectfont}%
  \ifx\svgwidth\undefined%
    \setlength{\unitlength}{157.68441028bp}%
    \ifx\svgscale\undefined%
      \relax%
    \else%
      \setlength{\unitlength}{\unitlength * \real{\svgscale}}%
    \fi%
  \else%
    \setlength{\unitlength}{\svgwidth}%
  \fi%
  \global\let\svgwidth\undefined%
  \global\let\svgscale\undefined%
  \makeatother%
  \begin{picture}(1,0.34870378)%
    \lineheight{1}%
    \setlength\tabcolsep{0pt}%
    \put(0,0){\includegraphics[width=\unitlength,page=1]{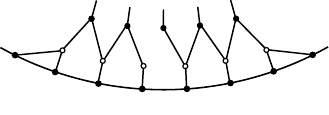}}%
    \put(0.4971908,-0.00811114){\color[rgb]{0,0,0}\rotatebox{-0.15281934}{\makebox(0,0)[lt]{\lineheight{0}\smash{\begin{tabular}[t]{l}\tiny{$a$}\end{tabular}}}}}%
    \put(0,0){\includegraphics[width=\unitlength,page=2]{generallemma.0v2.pdf}}%
  \end{picture}%
\endgroup%
}}    
    \end{align*}
\end{lemma} 
\begin{proof}

We begin by applying a wrench relation on the leftmost fork in the drawn portion of $X$. If $X_1$ and $X_2$ are the resulting tensor diagrams, then we know $\mathbf{X}^{FP} = \mathbf{X}^{FP}_1 - \mathbf{X}^{FP}$ where $X_1$ is in the middle and $X_2$ is on the right.
\begin{center}
\begin{tikzpicture}
\node(0) at (0,0) {
\begingroup%
  \makeatletter%
  \providecommand\color[2][]{%
    \errmessage{(Inkscape) Color is used for the text in Inkscape, but the package 'color.sty' is not loaded}%
    \renewcommand\color[2][]{}%
  }%
  \providecommand\transparent[1]{%
    \errmessage{(Inkscape) Transparency is used (non-zero) for the text in Inkscape, but the package 'transparent.sty' is not loaded}%
    \renewcommand\transparent[1]{}%
  }%
  \providecommand\rotatebox[2]{#2}%
  \newcommand*\fsize{\dimexpr\f@size pt\relax}%
  \newcommand*\lineheight[1]{\fontsize{\fsize}{#1\fsize}\selectfont}%
  \ifx\svgwidth\undefined%
    \setlength{\unitlength}{137.19378374bp}%
    \ifx\svgscale\undefined%
      \relax%
    \else%
      \setlength{\unitlength}{\unitlength * \real{\svgscale}}%
    \fi%
  \else%
    \setlength{\unitlength}{\svgwidth}%
  \fi%
  \global\let\svgwidth\undefined%
  \global\let\svgscale\undefined%
  \makeatother%
  \begin{picture}(1,0.61675694)%
    \lineheight{1}%
    \setlength\tabcolsep{0pt}%
    \put(0,0){\includegraphics[width=\unitlength,page=1]{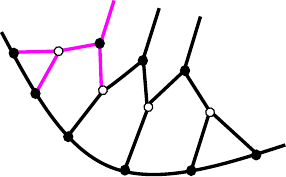}}%
  \end{picture}%
\endgroup%
};
\node(1) at (2.25,0) {$=$};
\node(2) at (4.5,0) {
\begingroup%
  \makeatletter%
  \providecommand\color[2][]{%
    \errmessage{(Inkscape) Color is used for the text in Inkscape, but the package 'color.sty' is not loaded}%
    \renewcommand\color[2][]{}%
  }%
  \providecommand\transparent[1]{%
    \errmessage{(Inkscape) Transparency is used (non-zero) for the text in Inkscape, but the package 'transparent.sty' is not loaded}%
    \renewcommand\transparent[1]{}%
  }%
  \providecommand\rotatebox[2]{#2}%
  \newcommand*\fsize{\dimexpr\f@size pt\relax}%
  \newcommand*\lineheight[1]{\fontsize{\fsize}{#1\fsize}\selectfont}%
  \ifx\svgwidth\undefined%
    \setlength{\unitlength}{137.19378374bp}%
    \ifx\svgscale\undefined%
      \relax%
    \else%
      \setlength{\unitlength}{\unitlength * \real{\svgscale}}%
    \fi%
  \else%
    \setlength{\unitlength}{\svgwidth}%
  \fi%
  \global\let\svgwidth\undefined%
  \global\let\svgscale\undefined%
  \makeatother%
  \begin{picture}(1,0.60734692)%
    \lineheight{1}%
    \setlength\tabcolsep{0pt}%
    \put(0,0){\includegraphics[width=\unitlength,page=1]{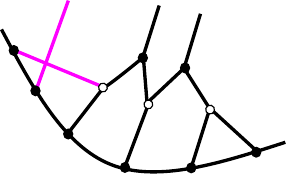}}%
  \end{picture}%
\endgroup%
};
\node(3) at (6.75,0) {$-$};
\node(4) at (9,0) {
\begingroup%
  \makeatletter%
  \providecommand\color[2][]{%
    \errmessage{(Inkscape) Color is used for the text in Inkscape, but the package 'color.sty' is not loaded}%
    \renewcommand\color[2][]{}%
  }%
  \providecommand\transparent[1]{%
    \errmessage{(Inkscape) Transparency is used (non-zero) for the text in Inkscape, but the package 'transparent.sty' is not loaded}%
    \renewcommand\transparent[1]{}%
  }%
  \providecommand\rotatebox[2]{#2}%
  \newcommand*\fsize{\dimexpr\f@size pt\relax}%
  \newcommand*\lineheight[1]{\fontsize{\fsize}{#1\fsize}\selectfont}%
  \ifx\svgwidth\undefined%
    \setlength{\unitlength}{137.19378374bp}%
    \ifx\svgscale\undefined%
      \relax%
    \else%
      \setlength{\unitlength}{\unitlength * \real{\svgscale}}%
    \fi%
  \else%
    \setlength{\unitlength}{\svgwidth}%
  \fi%
  \global\let\svgwidth\undefined%
  \global\let\svgscale\undefined%
  \makeatother%
  \begin{picture}(1,0.6488047)%
    \lineheight{1}%
    \setlength\tabcolsep{0pt}%
    \put(0,0){\includegraphics[width=\unitlength,page=1]{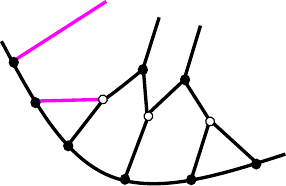}}%
  \end{picture}%
\endgroup%
};
\end{tikzpicture}
\end{center}

By Proposition \ref{prop:sign conversion}, Lemma \ref{lem: disjoint forks} and bilinearity, $\langle \mathbf{W},\mathbf{X}\rangle = \pm \langle \mathbf{W},\mathbf{X}^{FP}\rangle = \pm \langle \mathbf{W},\mathbf{X}_1^{FP}\rangle$. We proceed in this way counterclockwise across the boundary, using the wrench relation a total of $a+1$ times. After the final application, again one tensor diagram will give an invariant that pairs with $\mathbf{W}$ to 0 by Lemma \ref{lem: disjoint forks}. The other is as below, and the resulting invariant pairs to 0 with $\mathbf{W}$ by Lemma \ref{lem:DisconnectedConveniently}.
\begin{center}
\begingroup%
  \makeatletter%
  \providecommand\color[2][]{%
    \errmessage{(Inkscape) Color is used for the text in Inkscape, but the package 'color.sty' is not loaded}%
    \renewcommand\color[2][]{}%
  }%
  \providecommand\transparent[1]{%
    \errmessage{(Inkscape) Transparency is used (non-zero) for the text in Inkscape, but the package 'transparent.sty' is not loaded}%
    \renewcommand\transparent[1]{}%
  }%
  \providecommand\rotatebox[2]{#2}%
  \newcommand*\fsize{\dimexpr\f@size pt\relax}%
  \newcommand*\lineheight[1]{\fontsize{\fsize}{#1\fsize}\selectfont}%
  \ifx\svgwidth\undefined%
    \setlength{\unitlength}{155.10486152bp}%
    \ifx\svgscale\undefined%
      \relax%
    \else%
      \setlength{\unitlength}{\unitlength * \real{\svgscale}}%
    \fi%
  \else%
    \setlength{\unitlength}{\svgwidth}%
  \fi%
  \global\let\svgwidth\undefined%
  \global\let\svgscale\undefined%
  \makeatother%
  \begin{picture}(1,0.52303101)%
    \lineheight{1}%
    \setlength\tabcolsep{0pt}%
    \put(0,0){\includegraphics[width=\unitlength,page=1]{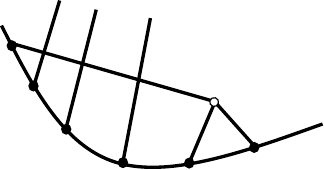}}%
  \end{picture}%
\endgroup%

\end{center}
\end{proof}
Using these lemmas and a few other coloring arguments, we give a list of specific boundary configurations which pair to 0. Each of these configurations appear multiple times when evaluating the pairings between $\SL_3$ and $\SL_4$ basis webs in the proof of Theorem \ref{thm:DualityComputation}.

\begin{lemma}\label{lem:pairing0r=4k=3}
Suppose $\mathbf{W}$ is an $\SL_4$ web invariant and $\mathbf{X}$ is an $\SL_3$ web invariant such that, for a fixed cyclic interval of boundary vertices and one item below, $W$ contains a boundary configuration as on the left and $X$ contains a boundary configuration as on the right. Then $\langle \mathbf{W},\mathbf{X} \rangle = \langle \mathbf{X},\mathbf{W} \rangle = 0$. The same can be said about reflections of any of these items. 

\newcounter{rowcntr}
\renewcommand{\therowcntr}{\alph{rowcntr}}

\newcolumntype{A}{>{\refstepcounter{rowcntr}(\therowcntr)}c}

\AtBeginEnvironment{tabular}{\setcounter{rowcntr}{0}}

\begin{center}
\begin{tabular}{A c}
    \label{lem'(a)}&$\left\langle \adjustbox{valign=c}{\ \includegraphics[scale=.75]{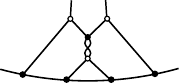}\ \ ,\ \ \includegraphics[scale=.75]{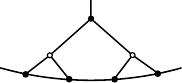}\ } \right\rangle$
\\
    \label{lem'(b)}&$\left\langle \adjustbox{valign=c}{\ \includegraphics[scale=.75]{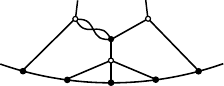}\ \ ,\ \ \includegraphics[scale=.75]{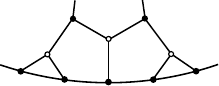}\ } \right\rangle$
\\
    \label{lem'(c)}&$\left\langle \adjustbox{valign=c}{\ \includegraphics[scale=.75]{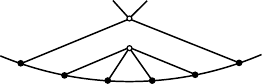}\ \ ,\ \ \includegraphics[scale=.75]{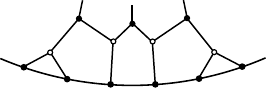}\ } \right\rangle$
 \\
    \label{lem'(e)}&$\left\langle \adjustbox{valign=c}{\ \includegraphics[scale=.75]{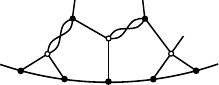}\ \ ,\ \ \includegraphics[scale=.75]{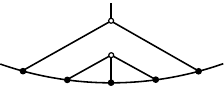}\ } \right\rangle$
\\
    \label{lem'(h)}&$\left\langle \adjustbox{valign=c}{\ \includegraphics[scale=.75]{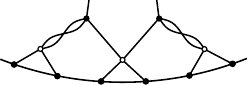}\ \ ,\ \ \includegraphics[scale=.75]{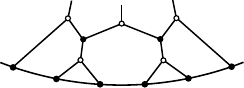}\ } \right\rangle$
\\
    \label{lem'(j)}&$\left\langle \adjustbox{valign=c}{\ \includegraphics[scale=.75]{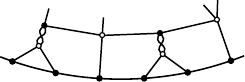}\ \ ,\ \ \includegraphics[scale=.75]{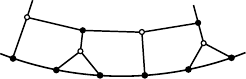}\ } \right\rangle$
\\
    \label{lem'(k)}&$\left\langle \adjustbox{valign=c}{\ \includegraphics[scale=.75]{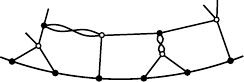}\ \ ,\ \ \includegraphics[scale=.75]{pairings0graphics/pairing7.1.pdf}\ } \right\rangle$
\\
    \label{lem'(d)}&$\left\langle \adjustbox{valign=c}{\ \includegraphics[scale=.75]{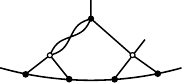}\ \ ,\ \ \includegraphics[scale=.75]{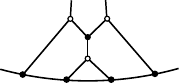}\ } \right\rangle$
\\
    \label{lem'(i)}&$\left\langle \adjustbox{valign=c}{\ \includegraphics[scale=.75]{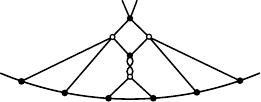}\ \ ,\ \ \includegraphics[scale=.75]{pairings0graphics/pairing6.1.pdf}\ } \right\rangle$
\\
    \label{lem'(g)}&$\left\langle \adjustbox{valign=c}{\ \includegraphics[scale=.75]{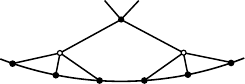}\ \ ,\ \ \includegraphics[scale=.75]{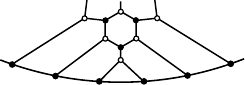}\ } \right\rangle$
\\
    \label{lem'(f)}&$\left\langle \adjustbox{valign=c}{\ \includegraphics[scale=.75]{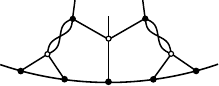}\ \ ,\ \ \includegraphics[scale=.75]{pairings0graphics/pairing4.1.pdf}\ } \right\rangle$
\\
    \label{lem'(l)}&$\left\langle \adjustbox{valign=c}{\ \includegraphics[scale=.75]{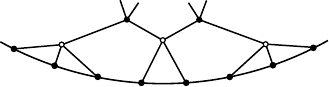}\ \ ,\ \ \includegraphics[scale=.75]{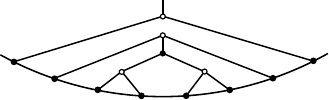}\ } \right\rangle$
\\
    \label{lem'(m)}&$\left\langle \adjustbox{valign=c}{\ \includegraphics[scale=.75]{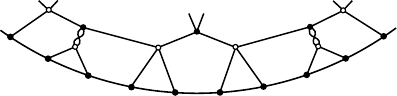}\ \ ,\ \ \includegraphics[scale=.75]{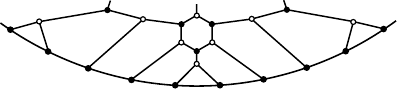}\ } \right\rangle$
\end{tabular}
\end{center}
\end{lemma}

\begin{proof}
We begin by noting that the symmetry in the pairing follows from Remark \ref{rmk:SymmetricFLL}. The fact that these are preserved by reflections follows from the symmetry in the arguments used. So, we will focus on the explicit pairings given in the list.

Items \eqref{lem'(a)}, \eqref{lem'(b)}, and \eqref{lem'(c)} follow immediately from Lemma \ref{lem:pairing0general}. 
Item \eqref{lem'(e)} is an immediate consequence of Lemma \ref{lem:DisconnectedConveniently}. 
Items \eqref{lem'(h)}, \eqref{lem'(j)}, and \eqref{lem'(k)} can be shown by applying repeated skein relations (on both $W$ and $X$) and invoking Lemmas \ref{lem: disjoint forks} and \ref{lem:DisconnectedConveniently}. 

To show Item \eqref{lem'(d)}, we begin by applying a wrench relation on $X$. This along with Proposition \ref{prop:sign conversion} gives us (from left-to-right below) $\mathbf{X} = \pm \mathbf{X}_1 \pm\mathbf{X}_2$. By Lemma \ref{lem: disjoint forks} and bilinearity, $\langle \mathbf{W},\mathbf{X}\rangle = \pm \langle \mathbf{W},\mathbf{X}_1\rangle$.

\begin{center}
\begin{tikzpicture}
    \node(0) at (0,0) {\scalebox{0.75}{
\begingroup%
  \makeatletter%
  \providecommand\color[2][]{%
    \errmessage{(Inkscape) Color is used for the text in Inkscape, but the package 'color.sty' is not loaded}%
    \renewcommand\color[2][]{}%
  }%
  \providecommand\transparent[1]{%
    \errmessage{(Inkscape) Transparency is used (non-zero) for the text in Inkscape, but the package 'transparent.sty' is not loaded}%
    \renewcommand\transparent[1]{}%
  }%
  \providecommand\rotatebox[2]{#2}%
  \newcommand*\fsize{\dimexpr\f@size pt\relax}%
  \newcommand*\lineheight[1]{\fontsize{\fsize}{#1\fsize}\selectfont}%
  \ifx\svgwidth\undefined%
    \setlength{\unitlength}{137.09623682bp}%
    \ifx\svgscale\undefined%
      \relax%
    \else%
      \setlength{\unitlength}{\unitlength * \real{\svgscale}}%
    \fi%
  \else%
    \setlength{\unitlength}{\svgwidth}%
  \fi%
  \global\let\svgwidth\undefined%
  \global\let\svgscale\undefined%
  \makeatother%
  \begin{picture}(1,0.7530862)%
    \lineheight{1}%
    \setlength\tabcolsep{0pt}%
    \put(0,0){\includegraphics[width=\unitlength,page=1]{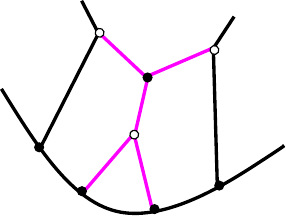}}%
  \end{picture}%
\endgroup%
}};    
    \node(2) at (2,0) {$=$};
    \node(3) at (4,0) {\scalebox{0.75}{
\begingroup%
  \makeatletter%
  \providecommand\color[2][]{%
    \errmessage{(Inkscape) Color is used for the text in Inkscape, but the package 'color.sty' is not loaded}%
    \renewcommand\color[2][]{}%
  }%
  \providecommand\transparent[1]{%
    \errmessage{(Inkscape) Transparency is used (non-zero) for the text in Inkscape, but the package 'transparent.sty' is not loaded}%
    \renewcommand\transparent[1]{}%
  }%
  \providecommand\rotatebox[2]{#2}%
  \newcommand*\fsize{\dimexpr\f@size pt\relax}%
  \newcommand*\lineheight[1]{\fontsize{\fsize}{#1\fsize}\selectfont}%
  \ifx\svgwidth\undefined%
    \setlength{\unitlength}{137.09623682bp}%
    \ifx\svgscale\undefined%
      \relax%
    \else%
      \setlength{\unitlength}{\unitlength * \real{\svgscale}}%
    \fi%
  \else%
    \setlength{\unitlength}{\svgwidth}%
  \fi%
  \global\let\svgwidth\undefined%
  \global\let\svgscale\undefined%
  \makeatother%
  \begin{picture}(1,0.82571581)%
    \lineheight{1}%
    \setlength\tabcolsep{0pt}%
    \put(0,0){\includegraphics[width=\unitlength,page=1]{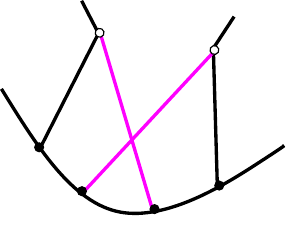}}%
    \put(0.50080519,0.00767776){\color[rgb]{0,0,0}\rotatebox{-0.15281934}{\makebox(0,0)[lt]{\lineheight{0}\smash{\begin{tabular}[t]{l}$j_1$\end{tabular}}}}}%
    \put(0.75312704,0.08994097){\color[rgb]{0,0,0}\rotatebox{-0.15281934}{\makebox(0,0)[lt]{\lineheight{0}\smash{\begin{tabular}[t]{l}$j_2$\end{tabular}}}}}%
    \put(0.204993,0.06292126){\color[rgb]{0,0,0}\rotatebox{-0.15281934}{\makebox(0,0)[lt]{\lineheight{0}\smash{\begin{tabular}[t]{l}$i_2$\end{tabular}}}}}%
    \put(0.02938483,0.22616899){\color[rgb]{0,0,0}\rotatebox{-0.15281934}{\makebox(0,0)[lt]{\lineheight{0}\smash{\begin{tabular}[t]{l}$i_1$\end{tabular}}}}}%
  \end{picture}%
\endgroup%
}};    
    \node(4) at (6,0) {$-$};
    \node(5) at (8,0) {\scalebox{0.75}{
\begingroup%
  \makeatletter%
  \providecommand\color[2][]{%
    \errmessage{(Inkscape) Color is used for the text in Inkscape, but the package 'color.sty' is not loaded}%
    \renewcommand\color[2][]{}%
  }%
  \providecommand\transparent[1]{%
    \errmessage{(Inkscape) Transparency is used (non-zero) for the text in Inkscape, but the package 'transparent.sty' is not loaded}%
    \renewcommand\transparent[1]{}%
  }%
  \providecommand\rotatebox[2]{#2}%
  \newcommand*\fsize{\dimexpr\f@size pt\relax}%
  \newcommand*\lineheight[1]{\fontsize{\fsize}{#1\fsize}\selectfont}%
  \ifx\svgwidth\undefined%
    \setlength{\unitlength}{137.09623682bp}%
    \ifx\svgscale\undefined%
      \relax%
    \else%
      \setlength{\unitlength}{\unitlength * \real{\svgscale}}%
    \fi%
  \else%
    \setlength{\unitlength}{\svgwidth}%
  \fi%
  \global\let\svgwidth\undefined%
  \global\let\svgscale\undefined%
  \makeatother%
  \begin{picture}(1,0.82571581)%
    \lineheight{1}%
    \setlength\tabcolsep{0pt}%
    \put(0,0){\includegraphics[width=\unitlength,page=1]{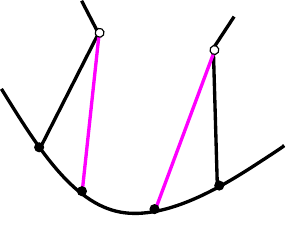}}%
    \put(0.50322023,0.00767776){\color[rgb]{0,0,0}\rotatebox{-0.15281934}{\makebox(0,0)[lt]{\lineheight{0}\smash{\begin{tabular}[t]{l}$j_1$\end{tabular}}}}}%
    \put(0.75312704,0.08028075){\color[rgb]{0,0,0}\rotatebox{-0.15281934}{\makebox(0,0)[lt]{\lineheight{0}\smash{\begin{tabular}[t]{l}$j_2$\end{tabular}}}}}%
    \put(0.20982311,0.07982662){\color[rgb]{0,0,0}\rotatebox{-0.15281934}{\makebox(0,0)[lt]{\lineheight{0}\smash{\begin{tabular}[t]{l}$i_2$\end{tabular}}}}}%
    \put(0.03421494,0.24307435){\color[rgb]{0,0,0}\rotatebox{-0.15281934}{\makebox(0,0)[lt]{\lineheight{0}\smash{\begin{tabular}[t]{l}$i_1$\end{tabular}}}}}%
  \end{picture}%
\endgroup%
}};    
\end{tikzpicture}
\end{center}

The tensor diagram $X_1$ has two forks. Call these $(i_1,j_1)$ and $(i_2,j_2)$. As in the proof of Lemma \ref{lemma: plu expansion fork}, reduce $X_1$ using wrench relations until in all resulting diagrams, $(i_1,j_1)$ is contained in a claw, $(i_1,j_1,\ell)$. We can do this in a way such that, in each resulting tensor diagram, either $(i_2,j_2)$ is in a fork or $\ell \in \{i_2,j_2\}$. For any tensor diagrams in the former set, we can ignore the claw $(i_1,j_1,\ell)$ and repeat the process with respect to the fork $(i_2,j_2)$. The end result is an expansion $\mathbf{X}_1 = \pm\mathbf{Y}_1 \pm \cdots \pm \mathbf{Y}_m$ such that each $Y_i$ is a product of Pl\"ucker coordinates and either $(i_1,j_1)$ and $(i_2,j_2)$ are both in a fork or three of the four boundary nodes $\{i_1,j_1,i_2,j_2\}$ are in a claw. If $Y_i$ is described by the former, we can observe that $\langle \mathbf{W},\mathbf{Y}_i \rangle = 0$ with a coloring argument. If $Y_i$ is described by the latter, $\langle \mathbf{W},\mathbf{Y}_i \rangle$ is immediately 0 by Lemma \ref{lem: disjoint forks}. 

Items \eqref{lem'(i)}  and \eqref{lem'(g)} can be shown using a similar tactic as was used for Item \eqref{lem'(d)}.

To prove Item \eqref{lem'(f)}, we use Lemma \ref{lemma: plu expansion fork} to rewrite $\mathbf{X}$ as a sum of products of Pl{\"u}ckers where the fork between the extreme vertices is preserved. Then, the pairing is shown to be 0 by a coloring argument. 

To show Items \eqref{lem'(l)} and \eqref{lem'(m)}, as in many previous computations, we use $\SL_3$ skein relations and Lemma \ref{lem: disjoint forks} until the remaining tensor invariant which possibly could have nonzero pairing with $\mathbf{W}$ has a diagram with several claws. Then, the statement holds by showing that these claws present an impossible coloring condition on $W$.
\end{proof}

We now illustrate how these lemmas can be combined to show that $\B^*$ and $\B$ are dual bases with respect to the FLL pairing. 

\begin{proof}[Proof of Theorem \ref{thm:DualityComputation}]
We begin by fixing an element $\mathbf{W}_i\in\B^*$ and evaluate $\langle \mathbf{W}_i, \rho^\ell(\mathbf{X}_j) \rangle$ for all distinct rotations for each $X_j \in B$.  The fact that it suffices to consider only one $\SL_4$ web in each rotation orbit follows from  Lemma \ref{lem:RotationBothWebs} and the fact that promotion and transpose commute, i.e., for all tableaux $S$, we have $(\rho(S))^t = \rho(S^t)$.
We claim that $\langle \mathbf{W}_i, \mathbf{X}_i \rangle = \sign(X_i)$, and $\langle \mathbf{W}_i, \rho^\ell(\mathbf{X}_j) \rangle = 0$ for all other cases.  Moreover, we observe the corresponding tableaux of $W_i$ and $X_i$ are the transpose of one another.

We include a set of representative computations. 
For each $\mathbf{W}_i$, we begin by using Lemma \ref{lem: disjoint forks} to eliminate a set of webs in $\B$ whose web invariants will evaluate to 0 when paired with $\mathbf{W}_i$. The presence of a fork between vertices $\ell$ and $\ell+1$ in web $X$ is equivalent to $\ell$ being in the descent set of the row word of the corresponding standard Young tableau \cite[Lemma 5.42]{GPPSS24}. Therefore, we can analyze descent sets using SageMath \cite{sage} and identify these webs $\rho^\ell(X_j)$ which have $\langle \mathbf{W}_i, \rho^\ell(\mathbf{X}_j) \rangle = 0$ by Lemma \ref{lem: disjoint forks}. 

For example, in the case of $W_{27}$ this process eliminates all basis webs except  $X_{27}$.  
Next, we must check that $\langle \mathbf{W}_{27},\mathbf{X}_{27} \rangle = \sign(X_{27})$, or equivalently,  $\langle \mathbf{W}_{27},\mathbf{X}_{27}^{\text{FP}} \rangle = 1$.  We use skein relations to write $\mathbf{X}_{27}^{\text{FP}}$ in terms of Pl\"ucker coordinates, 
\[\mathbf{X}_{27}^{\text{FP}} = (\Delta_{1,2,4}\Delta_{3,5,12} - \Delta_{1,2,3}\Delta_{4,5,12})(\Delta_{6,8,11}\Delta_{6,7,9} - \Delta_{6,7,11}\Delta_{8,9,10}).
\]
Using the bilinearity of the inner product, we have
\begin{align*}
\langle \mathbf{W}_{27},  \mathbf{X}_{27}^{\text{FP}} \rangle &= \langle \mathbf{W}_{27}, (\Delta_{1,2,4}\Delta_{3,5,12} - \Delta_{1,2,3}\Delta_{4,5,12})(\Delta_{6,8,11}\Delta_{7,9,10} - \Delta_{6,7,11}\Delta_{8,9,10}) \rangle \\
&= \langle \mathbf{W}_{27}, \Delta_{1,2,4}\Delta_{3,5,12}\Delta_{6,8,11}\Delta_{7,9,10}\rangle - \langle \mathbf{W}_{27}, \Delta_{1,2,4}\Delta_{3,5,12}\Delta_{6,7,11}\Delta_{8,9,10}\rangle \\
&- \langle \mathbf{W}_{27}, \Delta_{1,2,3}\Delta_{4,5,12}\Delta_{6,8,11}\Delta_{7,9,10}\rangle + \langle \mathbf{W}_{27}, \Delta_{1,2,3}\Delta_{4,5,12}\Delta_{6,7,11}\Delta_{8,9,10}\rangle. \\
\end{align*}
Any term with $\Delta_{1,2,3}$ will fail to properly color $W_{27}$ since the vertices $2$ and $3$ form a fork. Similarly, any term with $\Delta_{8,9,10}$ will fail to properly color $W_{27}$ as the vertices $8$ and $9$ form a fork. Hence, it suffices to check the inner product with the term $\Delta_{1,2,4}\Delta_{3,5,12}\Delta_{6,8,11}\Delta_{7,9,10}$. As shown below, there is a unique coloring of $W_{27}$ with this term. By Proposition \ref{proposition:PairingAsCountingLabelings}, we have that $\langle \mathbf{W}_{27}, \mathbf{X}_{27}^{\text{FP}}\rangle = \langle \mathbf{W}_{27}, \Delta_{1,2,4}\Delta_{3,5,12}\Delta_{6,8,11}\Delta_{7,9,10} \rangle = 1$. 

 $$\left\langle \adjustbox{valign=c}{\ \includegraphics[scale=.5]{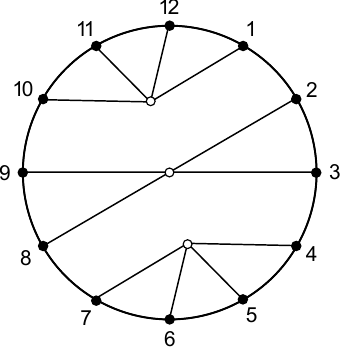}\ \ ,\ \ \includegraphics[scale=.5]{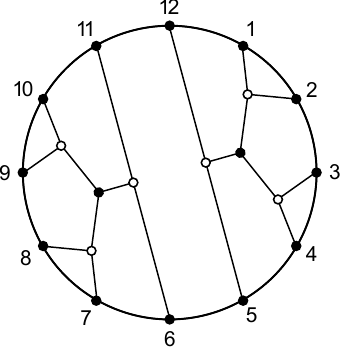}\ } \right\rangle
    \qquad = \left\langle \adjustbox{valign=c}{\ \includegraphics[scale=.5]{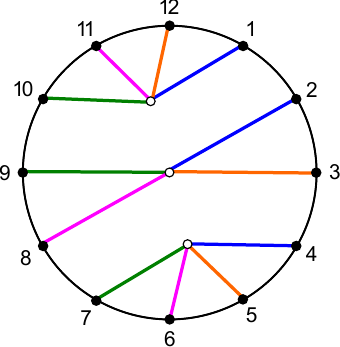}\ \ ,\ \ \includegraphics[scale=.5]{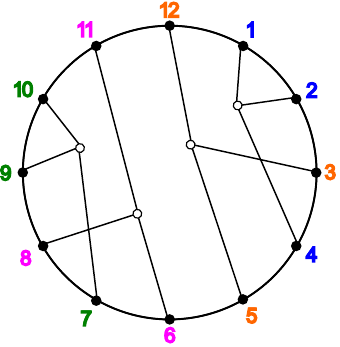}\ } \right\rangle
    \qquad $$
    
In general, for a web $W_i$, many but not necessarily all webs $\rho^\ell(X_j) \neq X_i$ are covered by Lemma \ref{lem: disjoint forks}. For instance,  when we perform the same process in SageMath \cite{sage} for web  $W_2$, we have 26 webs remaining whose web invariants which could pair nontrivially with $\mathbf{W}_2$. These are drawn below, where the web on the top left is $X_2$. 

\begin{table}[H]
    \centering
    \begin{tabular}{cccccc}
\includegraphics[scale = 0.4]{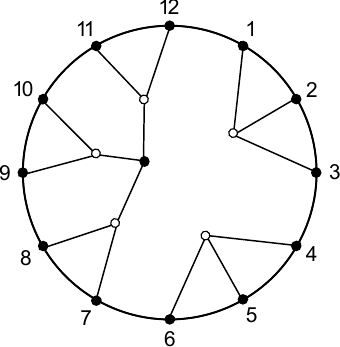} & 
\includegraphics[scale =0.4]{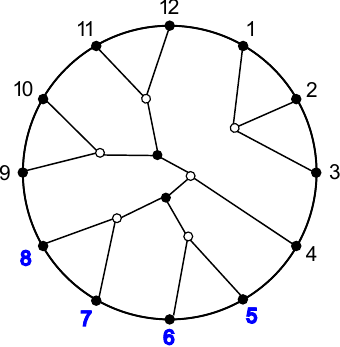} & 
 \includegraphics[scale = 0.4]{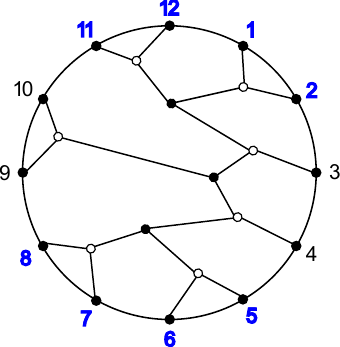} &
\includegraphics[scale = 0.4]{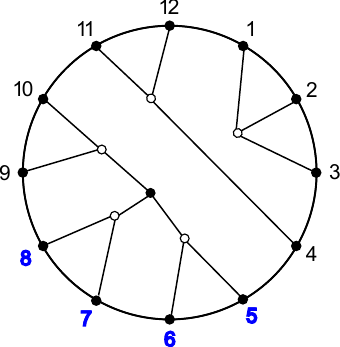} & 
 \includegraphics[scale = 0.4]{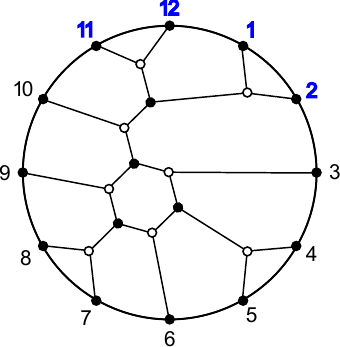} &
 \includegraphics[scale = 0.4]{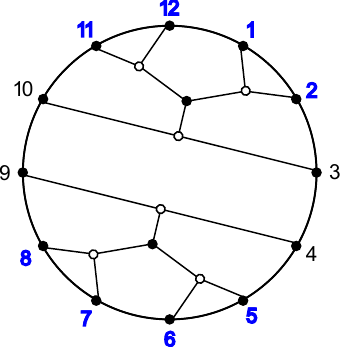} \\

\includegraphics[scale = 0.4]{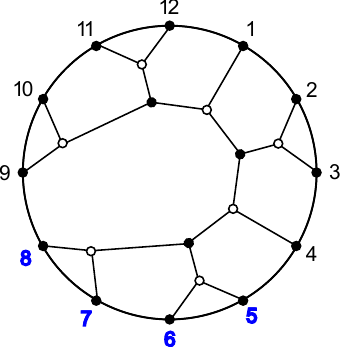} & 
\includegraphics[scale = 0.4]{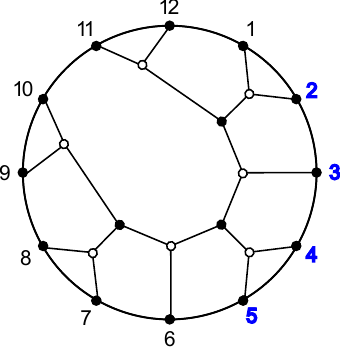} &
\includegraphics[scale = 0.4]{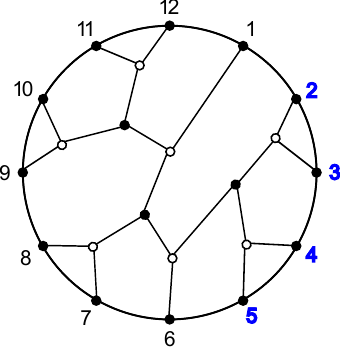} & \includegraphics[scale = 0.4]{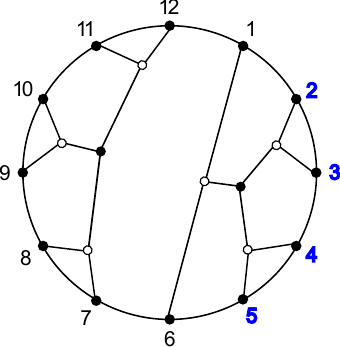} &
\includegraphics[scale = 0.4]{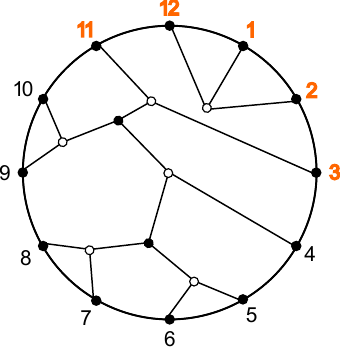} & 
\includegraphics[scale = 0.4]{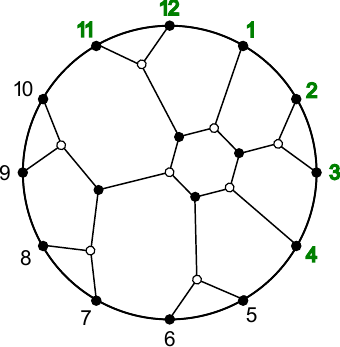} \\

\includegraphics[scale = 0.4]{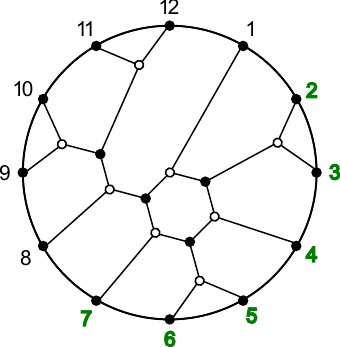} &
\includegraphics[scale = 0.4]{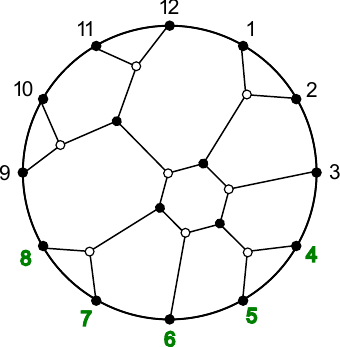} &
\includegraphics[scale = 0.4]{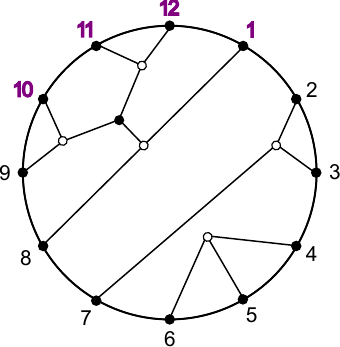}   &
\includegraphics[scale = 0.4]{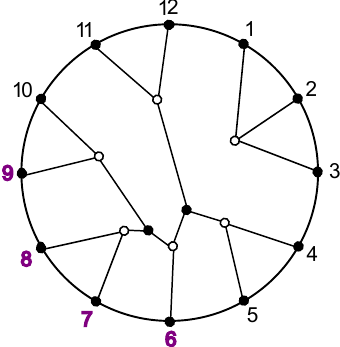}  & 
\includegraphics[scale = 0.4]{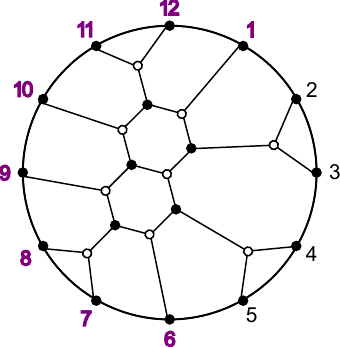} & 
\includegraphics[scale = 0.4]{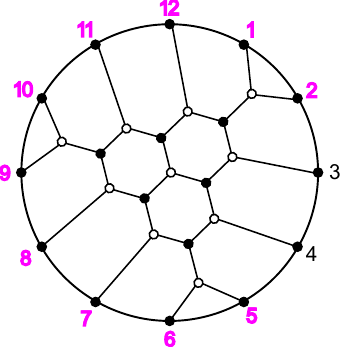} \\
    \end{tabular}
    \caption{All $\SL_3$ basis webs for which we cannot use Lemmas \ref{lem: disjoint forks}, \ref{lem:DisconnectedConveniently}, \ref{lem:pairing0general}, or \ref{lem:pairing0r=4k=3} when pairing with $W_2$.}
    \label{tab:my_label}
\end{table}

We may apply Lemma \ref{lem:pairing0r=4k=3}(a) to the webs with blue vertices to demonstrate that $\mathbf{W}_2$ pairs trivially with these webs. Note the blue vertices indicate a boundary configuration where this lemma can be applied. We similarly use item (d) of the same lemma on the web with orange vertices, item (f) on the webs with green vertices, item (h) on the webs with purple vertices,  and item (m) on the web with pink vertices. 
None of these lemmas applied to web $X_2$. In a similar manner to previously, we can compute $\langle \mathbf{W}_2,\mathbf{X}_2 \rangle = \sign(X_2)$. 

 Finally, we note that there are pairs $(W_i,\rho^\ell(X_j))$ with $\rho^\ell(X_j) \neq X_i$ which are not covered by our Lemmas. An example occurs when pairing $\mathbf{W}_1$ with $\rho^3(\mathbf{X}_{30})$. As in the proof of Lemma \ref{lem:DisconnectedConveniently}, we apply a series of wrench relations until we have written $\rho^3(\mathbf{X}_{30})$ as a linear combination of products of Pl{\"u}cker coordinates. Many of the resulting products pair to 0 with $\mathbf{W}_1$ by one of the previous lemmas. For the remaining products, we can verify that they still provide an impossible coloring condition for $W_1$.
 
Finally, consider $\mathbf{W}_{16}$, which is the difference of two web invariants. Let these webs be $W^{ben}$ and $W^{dis}$ which are the connected and disconnected webs respectively. Lemma \ref{lem: disjoint forks} guarantees that $\langle \mathbf{W}^{ben},\mathbf{X} \rangle = \langle \mathbf{W}^{dis},\mathbf{X}\rangle = 0$ for all but six $\SL_3$ basis webs; these six are drawn below. 

\begin{center}
\begin{tabular}{ccc}
 \includegraphics[width = 1.25 in]{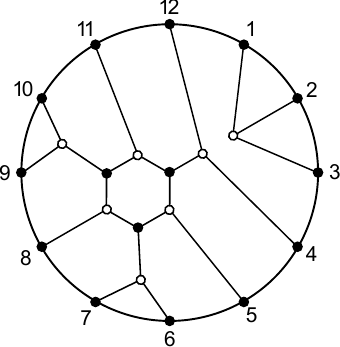} & \includegraphics[width = 1.25 in]{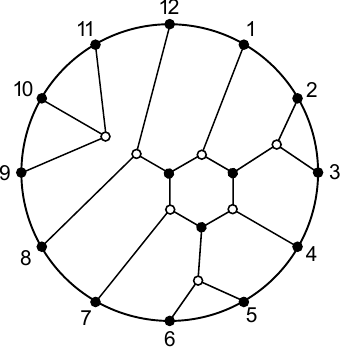}  & \includegraphics[width = 1.25 in]{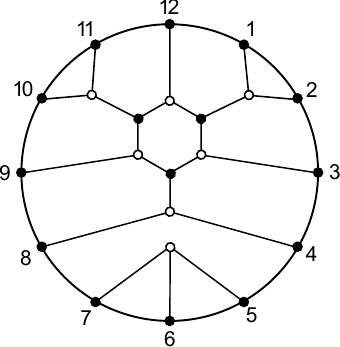}  \\
  \includegraphics[width = 1.25 in]{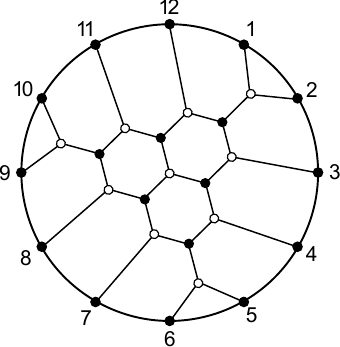}   & \includegraphics[width = 1.25 in]{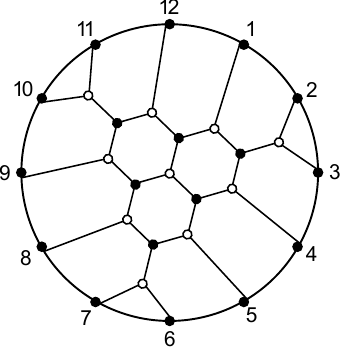}  & \includegraphics[width = 1.25 in]{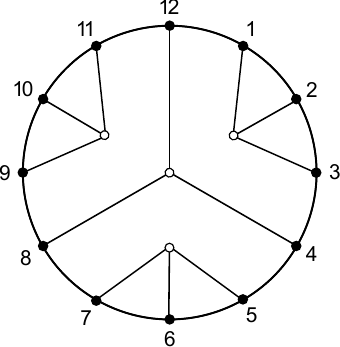}
\end{tabular}
\end{center}

The invariants associated to all three rotations of $X_{20}$ pair to 0 with $\mathbf{W}^{dis}$ again by Lemma \ref{lem: disjoint forks} and pair to 0 with $\mathbf{W}^{ben}$ by Lemma \ref{lem:pairing0general} with $r = 4, a = 3, b = 2$. Indeed, since $W^{ben}$ and $W^{dis}$ are invariant under $2\pi/3$ rotation, we need to only check the pairing with $\mathbf{X}_{20}$ by Lemma \ref{lem:RotationBothWebs}. 
By using $\SL_3$ skein relations and Proposition \ref{proposition:PairingAsCountingLabelings}, we can compute $\langle \mathbf{W}^{ben}, \mathbf{X}_{32} \rangle = \langle \mathbf{W}^{dis}, \mathbf{X}_{32} \rangle = 0$ and $\langle \mathbf{W}^{ben}, \rho^3(\mathbf{X}_{32}) \rangle = \langle \mathbf{W}^{dis}, \rho^3(\mathbf{X}_{32}) \rangle = 1$ implying $\langle \mathbf{W}_{16},\mathbf{X}_{32} \rangle = 0$, and similarly for $\rho^3(\mathbf{X}_{32})$. Finally, by Proposition \ref{proposition:PairingAsCountingLabelings}, it is immediate that $\langle \mathbf{W}^{ben}, \mathbf{X}_{16} \rangle = 1$ and $\langle \mathbf{W}^{dis}, \mathbf{X}_{16} \rangle = 0$ so that  $\langle \mathbf{W}_{16},\mathbf{X}_{16} \rangle = \sign(X_{16}) = 1$, as desired.
\end{proof}

For all but four of the 462 basis webs, our dual basis $\B^*$ for $\C[\Gr(4,12)]_\lambda$ coincides with the basis given by the growth algorithm of \cite[Sec.\ 5]{GPPSS24}. The exceptional cases are the four rotations of the highly symmetric $\SL_3$ web $X_{16}=\Delta_{1,2,3}\Delta_{5,6,7}\Delta_{9,10,11}\Delta_{4,8,12}$.

Under the bijection of \cite{GPPSS24}, the transposed tableau $T(X_{16})^t$ corresponds to a web with a hexagonal ``benzene face'' where the edges alternate between single and double edges (the associated invariant web is one of the summands of $\mathbf{W}_{16}$). As explained in \cite[Sec.\ 6]{GPPSS24}, tableau $T(X_{16})^t$ in fact corresponds to the entire benzene-move-equivalence class of this web, where \emph{benzene moves} interchange the single and double edges bounding the hexagon. The authors of \cite{GPPSS24} break this symmetry by choosing the \emph{top} representative (maximal among the benzene-move-equivalence class). However, this is not a reflection-invariant choice, losing a property for $k=4$ that held for web bases when $k=2$ or $k=3$. By setting $\mathbf{W}_{16} = \mathbf{X}_{16}^*$ to be the difference of two webs one can obtain a basis which is both rotation- and reflection-invariant, as $W_{16}$ and its reflection in fact yield the same web invariant.

As an application of Theorems \ref{thm:TwistGivenBySumOfPairings} and \ref{thm:DualityComputation}, we obtain combinatorial expansion formulas for twists of functions in $\C[\Gr(3,12)]_\lambda$ and $\C[\Gr(4,12)]_\lambda$ ($\lambda=(1^{12})$) in terms of higher dimer covers.
In general, given a $\SL_r$-web $D$ and a web basis $\mathcal{W}$, let $\{C_W^D\}_{W \in \mathcal{W}}$ denote the coefficients of the expansion of $\mathbf{D}$ in terms of $\mathcal{W}$, i.e. $\mathbf{D} = \sum_{W \in \mathcal{W}} C_W^D \mathbf{W}.$
\begin{theorem}
\label{thm:twist}
Let $n = 12$, $k \in \{3,4\}$, and $\lambda=(1^{12})$. Let $\mathbf{Y} \in \C[\Gr(k,12)]_\lambda$ be the web invariant associated with an $\SL_k$-basis web $Y$.
Let $G$ be a reduced top cell plabic graph of type $(k,n)$. Then, we can express 
$\tau(\mathbf{Y}) = \sign(Y) \sum_{D \in \mathcal{D}_r(G)} C_{Y^*}^D \fwt_G(D)$,
where $Y^*$ is the\footnote{When $Y = X_{16}$, which we refer to as the ``Benzene case," the dual web $Y^*$ is actually difference of two webs. The expansion formula still holds - taking either Benzene representative.}  dual basis web associated to $Y$. 
\end{theorem}

\begin{proof}
Let $r = \frac{12}{k}$ and $\mathcal{W}$ be the set of $\SL_r$ rotation-invariant basis webs for $\mathcal{W}_\lambda(\mathbb{C}^r)$. By first applying Theorem \ref{thm:TwistGivenBySumOfPairings} and then Theorem \ref{thm:DualityComputation}, we have 
$$\tau(\mathbf{Y}) = \sum_{D \in \mathcal{D}_{r,\lambda}(G)} \langle \mathbf{D},\mathbf{Y} \rangle \fwt_G(D) = \sum_{D \in \mathcal{D}_{r,\lambda}(G)} \sum_{W \in \mathcal{W}} C_W^D \langle \mathbf{W},\mathbf{Y} \rangle \fwt_G(D) = \sign(Y) \sum_{D \in \mathcal{D}_{r,\lambda}(G)} C_{Y^*}^D \fwt_G(D),$$
where $Y^*$ is the $\SL_r$-web dual to $Y$.
\end{proof}

In \cite{EMW23}, the authors provide several expansion formulas akin to Theorem \ref{thm:twist} for cluster variables associated to several $\Gr(3,6),\Gr(3,8)$ and $\Gr(3,9)$ webs. 

\begin{example}
Let $G$ be the reduced top cell plabic graph of type $(3,12)$  appearing in Figure \ref{fig:top cell plabic graph}.
Consider the basis web invariant $\mathbf{X}_{28}^\text{FP} \in \mathbb{C}[\Gr(3,12)]_{(1^{12})}$ associated to the Pl\"{u}cker degree four $\SL_3$ web $X_{28}$ (see Table \ref{tablep2}) in the FP sign convention.
We can expand $\mathbf{X}^\text{FP}_{28}$ in terms of Pl\"ucker coordinates 
using the wrench relation,
\begin{eqnarray*}\mathbf{X}_{28}^\text{FP} &=& (\Delta_{1,3,4}\Delta_{2,7,8} - \Delta_{1,7,8}\Delta_{2,3,4})
(\Delta_{5,9,10}\Delta_{6,11,12} - \Delta_{5,11,12}\Delta_{6,9,10})\\ 
&-& (\Delta_{1,3,4}\Delta_{2,5,6} - \Delta_{1,5,6}\Delta_{2,3,4})
(\Delta_{7,9,10}\Delta_{8,11,12} - \Delta_{7,11,12}\Delta_{8,9,10})
\end{eqnarray*}
We then use \cite[Prop. 4.3]{EMW23} to write its twist, 
$$\tau(\mathbf{X}_{28}^\text{FP}) ~~=~~ [\Delta_{1,2,12}\Delta_{2,3,4}\Delta_{4,5,6}\Delta_{6,7,8}\Delta_{8,9,10}\Delta_{10,11,12}] (\Delta_{1,7,11}\Delta_{3,5,9} - \Delta_{1,9,11}\Delta_{3,5,7})$$
which is a Pl\"{u}cker degree two cluster variable multiplied by a Pl\"{u}cker degree six frozen monomial.  
Thus, as an application of Theorem \ref{thm:twist}, the Laurent expansion 
$$\tau(\mathbf{X}_{28}^\text{FP}) ~~=~~ 
\frac{\Delta_{1,2,12}\Delta_{2,3,4}\Delta_{4,5,6}\Delta_{6,7,8}\Delta_{8,9,10}\Delta_{10,11,12}(\Delta_{1,3,5}\Delta_{1,9,11}\Delta_{5,7,9} + \Delta_{1,5,11}\Delta_{1,7,9}\Delta_{3,5,9})}{\Delta_{1,5,9}}$$
relative to this choice of initial seed agrees with the face weights of the only two quadruple dimer covers of this plabic graph, as illustrated in Figure \ref{fig:QuadEx} (top), that admit a $4$-weblike subgraph, see Figure \ref{fig:QuadEx} (bottom), that expands via $\SL_4$ skein relations
to include $W_{28}$, the dual of $X_{28}$, in its support. In this case, in each expansion, $W_{28}$ appears with coefficient 1. 
Note that the $4$-weblike subgraph on the right has forbidden squares in its move equivalence class which allow us to apply the $\SL_4$-skein relations. 
\end{example}

In light of Theorem~\ref{thm:DualityComputation}, and the observations in the Appendix in \cite{FLL19}, it is natural to hope for generalizations. Given $n,k,r$ with $n = kr$ and a $k\times r$ standard Young tableau $T$ and its transpose $T^t$, are there corresponding web basis elements $W$ and $W^*$, related by web duality, associated to $\C[\Gr(k,n)]$ and $\C[\Gr(r,n)]$, respectively, such that the Laurent expansion of the Pl\"{u}cker degree $r$ $\SL_k$-web invariant $\mathbf{W}$ in the cluster algebra for $\mathbb{C}[\Gr(k,n)]$ may be expressed as in Theorem \ref{thm:TwistGivenBySumOfPairings} using the dual web $W^*$? Or, similar to our web 16 in Table \ref{tablep1}, are there other examples where one basis web is dual to a linear combination of basis webs? The difficulty heretofore of producing web bases possessing natural symmetries (such as those of \cite{Weyl1932, kup96, GPPSS24, fraser23}) suggests that these may be hard questions.

\begin{figure}
\begin{center}
\begin{tabular}{cc}  
 \includegraphics[width=.335\textwidth]{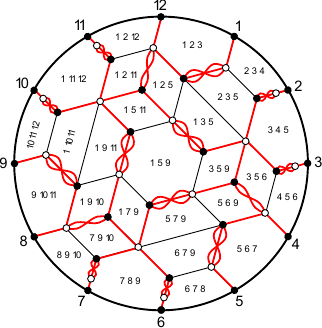}
&
\includegraphics[width=.335\textwidth]{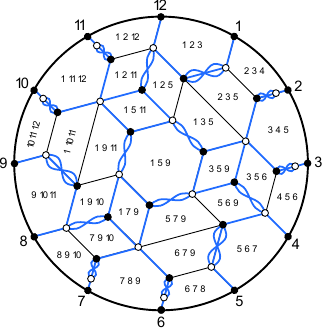} \\
\includegraphics[width=.22\textwidth]{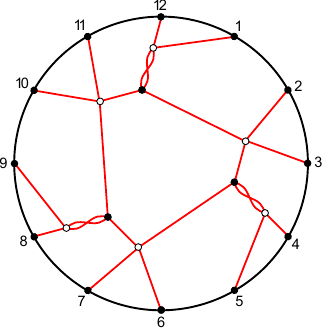} &
\includegraphics[width=.22\textwidth]{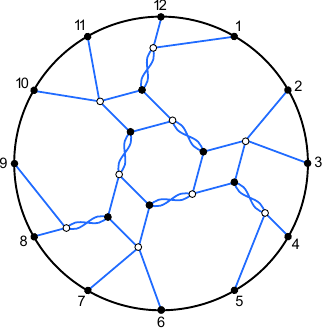}
\end{tabular}
\caption{(Top): Two $4$-dimer covers on the plabic graph $G$ for $\mathbb{C}[\Gr(3,12)]$; \hspace{7em} (Bottom): The $4$-weblike subgraphs equivalent to these choices of $4$-dimer covers.}
\label{fig:QuadEx}
\end{center}
\end{figure}



\section{Enumeration of cluster variables}
\label{sec:Counting} 

In \cite{CDHHHL22}, the authors use high-performance computing to conjecture the number of cluster variables of small Pl\"{u}cker degrees in $\mathbb{C}[\Gr(3,n)]$ and $\mathbb{C}[\Gr(4,n)]$.
In particular, letting $N_{k,n,r}$ be the number of cluster variables of Pl\"{u}cker degree $r$ in $\mathbb{C}[\Gr(k,n)]$, the authors conjecture\footnote{The 52 is a 48 in \cite[Conjecture 3.1]{CDHHHL22}, but has been confirmed by one of the authors of \cite{CDHHHL22} to be a typo.}
\begin{align}
    N_{3,n,4} &= 288\tbinom{n}{9} + 400\tbinom{n}{10} + 264\tbinom{n}{11} + 52\tbinom{n}{12}; \label{eq:N_3,n,4}\\
    N_{4,n,3} &= 174\tbinom{n}{9} + 855\tbinom{n}{10} + 1285\tbinom{n}{11} +123\tbinom{n}{12}. \label{eq:N_4,n,3}
\end{align}
Here, we compare this computational conjecture with that of Fomin--Pylyavskyy \cite{FP16}, which we now recall. 
 
We can define the unclasping of a tensor diagram as we did for webs in Definition~\ref{def:unclasping}. A tensor diagram whose unclasping is connected and contains no cycles is a \emph{tree diagram}. Fomin and Pylyavskyy introduced the \emph{arborization algorithm}, which takes as input any $\SL_3$ tensor diagram and outputs the result of applying arborizing steps (\cite[Definition 10.3]{FP16}) as many times as possible. See Example \ref{example:arborization} for illustrations of the arborization algorithm. A web $W$ is \emph{arborizable} if the result of applying the arborization algorithm to $W$ is a tree diagram; in this case we also say that the associated web invariant is arborizable. By \cite[Theorem 10.5]{FP16}, the arborization algorithm is confluent, so this notion is well-defined. 

\begin{example}\label{example:arborization}

We demonstrate the process of clasping web $X_{29}$ at vertices $\{1,12\}$ and at vertices $\{10,11\}$ and then applying the arborization algorithm. We only apply one arborizing step (corresponding to the left part of \cite[Figure 27]{FP16}). The resulting tensor diagram is a tree diagram since it is a tree if we unclasp all vertices. Therefore, we say the clasping of $X_{29}$ is \emph{arborizable}. Counting rotations and reflections of this clasped web, there are 20 webs in the entire dihedral orbit.

Similarly, we clasp vertices $\{1,11,12\}$ in web $X_{32}$. Applying the wrench relation three times yields a tree diagram, so this clasping of web $X_{32}$ is arborizable. There are also 20 webs in its dihedral orbit.

\begin{figure}[H]

\newcommand{\web}[1]{\adjustbox{valign=c}{\includegraphics[width=.15\textwidth]{webs/#1.pdf}}}

\[
\web{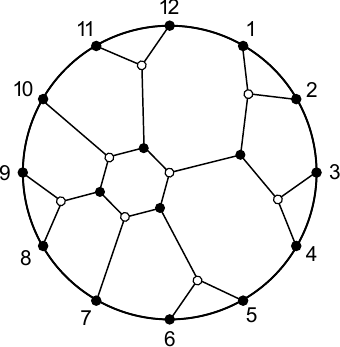}\ \xrightarrow[\{1,12\},\{10,11\}]{\text{clasp }}\ 
\web{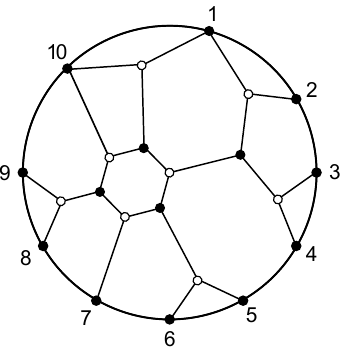}\ =\ 
\web{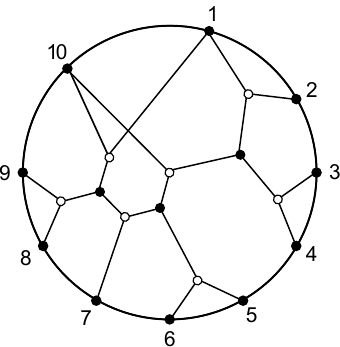} \ \xrightarrow{\text{unclasp}}\ 
\web{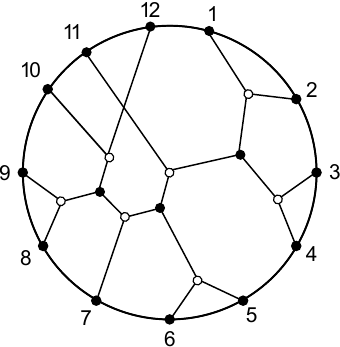}
\]
\begin{align*}
\web{_3,12_web43}\ \xrightarrow[\{1,11,12\}]{\text{clasp }}\ 
\web{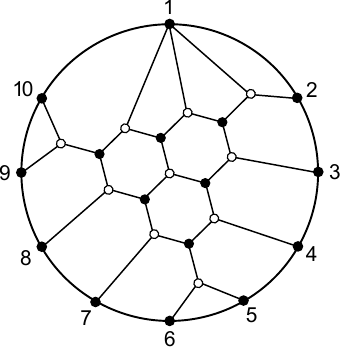}
\ =\ \web{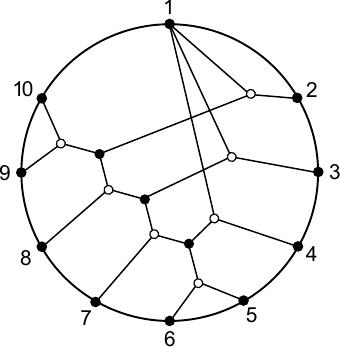} \ \xrightarrow{\text{unclasp}}\ \web{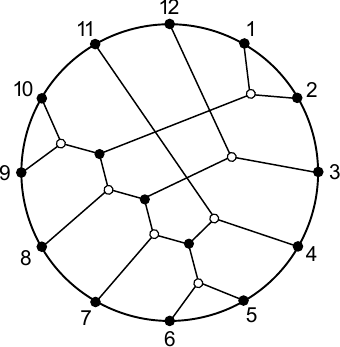}
 \end{align*}

\caption{The arborization algorithm applied to claspings of webs $X_{29}$ and $X_{32}$.}
\label{fig:arborization}
\end{figure}

\end{example}

Example \ref{example:arborization}  highlights two interesting features to the process of arborization. First, we see that in order to reach a tree, it was necessary to have a non-planar diagram. Moreover, even though $X_{29}$ and $X_{32}$ have internal cycles, the arborizing steps ``pushed'' these cycles onto the boundary. Therefore, when computing all arborizable clasps of $\Gr(3,12)$ webs, we must consider those with internal cycles. 

A web invariant is \emph{indecomposable} if it is not the product of two nontrivial web invariants. For example, $\mathbf{X}_{22}$ is decomposable since the web $X_{22}$ has two connected components, while $\mathbf{X}_{23}$ is indecomposable (see Remark \ref{rem:props of FP signs}). 

\begin{definition}
    Let $\mathscr{T}_{3,n}\subset \C[\Gr(3,n)]$ denote the set of invariants of all arborizable, indecomposable, non-elliptic semistandard $\SL_3$ webs. Let $T_{3,n,r}\coloneq\#\{\mathbf{X}\in\mathscr{T}_{3,n}\mid \mathbf{X} \text{ has Pl\"ucker degree }r\}.$
\end{definition}

Fomin--Pylyavskyy conjectured \cite[Conjectures 9.1, 10.1, 10.6]{FP16} that the set of cluster and coefficient variables in $\mathbb{C}[\Gr(3,n)]$ is exactly $\mathscr{T}_{3,n}$. In the special case when a standard web is already a tree diagram, the associated web invariant is known to be a cluster variable (see also \cite{le-yildirim} for evidence in support of these conjectures):

\begin{theorem}[\cite{FP16}, Corollary 8.10] \label{thm:TreesAreClusterVariables}
Let $X$ be a standard $\SL_3$ web which is a tree diagram. Then $\mathbf{X}$ is a cluster or coefficient variable in $\C[\Gr(3,n)]$. 
\end{theorem}

By enumerating all arborizable, indecomposable semistandard $\SL_3$ basis webs of Pl\"{u}cker degree 4, we have found that the conjectured formula \eqref{eq:N_3,n,4} of \cite{CDHHHL22} is consistent with the Fomin--Pylyavskyy conjectures, providing evidence in favor of both conjectures.

\begin{theorem}\label{thm:enum-of-arborizable-webs} 
    $T_{3,n,4}=288\tbinom{n}{9} + 400\tbinom{n}{10} + 264\tbinom{n}{11} + 52\tbinom{n}{12}$ (compare with Equation \eqref{eq:N_3,n,4}).
\end{theorem}

\begin{proof}
    In the case of a standard web, i.e.\ each boundary vertex is adjacent to exactly one edge, no arborization steps can be applied. This means that a standard web is arborizable and indecomposable if and only if it is a tree diagram. In the web basis for $\C[\Gr(3,12)]_{(1^{12})}$, the only trees are the dihedral orbits of webs $X_{23}$, $X_{24}$, $X_{25}$, and $X_{28}$, as shown in Table \ref{tablep2}. There are exactly 52 distinct rotations and reflections of these webs, matching the coefficient of ${n\choose 12}$ in \eqref{eq:N_3,n,4}. For each $X$ among these 52 webs and any $n \geq 12$, there are $52{n\choose 12}$ webs for $\C[\Gr(3,n)]$ obtained by attaching $X$ onto the boundary vertices $\{a_1,a_2,\ldots,a_{12}\} \in {[n]\choose 12}$. In particular, the remaining $n-12$ boundary vertices are not incident to any edges.

    To obtain the coefficients of ${n\choose 11}$, ${n\choose 10}$, and ${n\choose 9}$, we clasp adjacent boundary vertices of the indecomposable $\SL_3$ webs in Tables \ref{tablep1} and \ref{tablep2} to obtain webs for $\C[\Gr(3,11)]$, $\C[\Gr(3,10)]$, or $\C[\Gr(3,9)]$, and check whether the resulting web is arborizable and indecomposable. We illustrate this process in Example \ref{example:arborization}, and provide the full list of arborizable indecomposable webs in Appendix \ref{appendix:clasps}.
    As one simplifying step, note that if a web $X$ is arborizable, and if clasping adjacent boundary vertices of $X$ produces a boundary 4-cycle, then the resulting web is decomposable. Finally, we note that every web of Pl\"ucker degree 4 which has 8 or fewer boundary vertices is decomposable\footnote{This is consistent with the fact that $\C[\Gr(3,n)]$ is of finite type for $n \le 8,$ and in these cases there are no cluster variables of Pl\"ucker degree 4.}.
\end{proof}

\begin{remark}
    Fraser \cite[Theorem 9.10]{fraser20} showed that every cluster variable in $\C[\Gr(3,9)]$ (which is of finite mutation type) is an arborizable indecomposable non-elliptic web invariant. The authors of \cite{CDHHHL22} computed 288 cluster variables of Pl\"ucker degree 4 in $\C[\Gr(3,9)]$, and we find that $T_{3,9,4}=288$ (Theorem \ref{thm:enum-of-arborizable-webs}). We therefore conclude that in Pl\"ucker degree 4, the set of cluster variables in $\C[\Gr(3,9)]$ is exactly the set of arborizable indecomposable non-elliptic web invariants. One representative web in each dihedral orbit is illustrated in the middle column of Table \ref{table:Gr(3,9)clvars}.
\end{remark}

We now restrict to arborizable, indecomposable webs which are \emph{standard}, while generalizing to $k\ge 3$.

\begin{definition}
    Let $\mathscr{T}_{k,n}^\text{std}\subset \C[\Gr(k,n)]$ denote the set of invariants of all indecomposable, standard $\SL_k$ webs which are trees (possibly with edges of multiplicity $>1$). We also set $T_{k,n,r}^\text{std}\coloneq\#\{\mathbf{X}\in\mathscr{T}_{k,n}^\text{std}\mid \mathbf{X} \text{ has Pl\"ucker degree }r\}.$
\end{definition}

$\mathscr{T}_{k,n}^\text{std}$ is in bijection with the set of standard tree diagrams which have no 2-valent vertices. 

\begin{remark} When $k=4$, from Tables \ref{tablep1} and \ref{tablep2} we observe that the only trees in the web basis for $\C[\Gr(4,12)]_{(1^{12})}$ are the dihedral orbits of webs $W_{12}$, $W_{13}$, $W_{14}$, $W_{17}$, $W_{22}$, $W_{24}$, $W_{29}$, and $W_{31}$, totaling 123 webs. 
That is, 
\[T_{4,12,3}^{\text{std}}=123.\] This matches the coefficient of ${n\choose 12}$ in Equation \eqref{eq:N_4,n,3}. 
\end{remark}

In the case when $k=3$, we obtain a closed formula for $T_{3,3r,r}^\text{std}$. By Theorem \ref{thm:TreesAreClusterVariables}, $T_{3,3r,r}^\text{std}$ is a lower bound for the number of cluster variables in $\C[\Gr(3,3r)]_{(1^{3r})}$, and by the Fomin--Pylyavskyy conjectures these are expected to be equal.

\begin{proposition}\label{prop:counting_trees} 
For $r \geq 1$, we have $\displaystyle T_{3,3r,r}^\text{std} = {4r-3\choose r-1}\dfrac{2}{3r-1}.$ \footnote{This sequence is \href{https://oeis.org/A069271}{OEIS A069271}, with indexing shifted by 1.}
\end{proposition}

\begin{proof}
Let $X$ be a standard $\SL_3$ web of Pl\"ucker degree $r$ which is a tree diagram on $n=3r$ boundary vertices. The web $X$ gives rise to a binary tree $B$ as follows: delete the boundary vertex labeled by $n$ and its incident edge, let the white vertex adjacent to $n$ become the root, and unwrap the web, preserving planarity. This is easily seen to be a bijection between the set $\{\mathbf{X}\in\mathscr{T}_{3,n}^\text{std}\mid \mathbf{X} \text{ has Pl\"ucker degree }r\}$ and the set of full, ordered binary trees with $3r-1$ leaves (that is, $3r-2$ internal vertices), which are bipartite, such that the root is white and all leaves are black. From each such binary tree $B$, delete the (white) root vertex and its two incident edges, and contract the edge between each remaining white vertex and its parent black vertex. This produces an ordered pair $Q$ of full, ordered 4-ary trees. Figure \ref{fig:webs&trees} illustrates this algorithm for two examples. In both $X$ and $B$, there are $3r-2$ total internal vertices, and the exceedance is $r$. So the number of internal black vertices in $X$ and $B$ is $\frac{1}{2}(3r-2-r)=r-1$. This is the total number of internal vertices in $Q$, so $\{\mathbf{X}\in\mathscr{T}_{3,n}^\text{std}\mid \mathbf{X} \text{ has Pl\"ucker degree }r\}$ is in bijection with the set of ordered pairs of full, ordered 4-ary trees with $r-1$ total internal vertices. This set is enumerated by \href{https://oeis.org/A069271}{OEIS A069271}, with the indexing shifted by 1 \cite{Selkirk19}. 

\end{proof}

\begin{figure}[H]
    \begin{tabular}{c c c}
    \includegraphics[scale = 0.6]{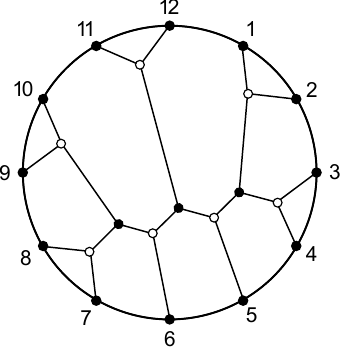}
    &
    \quad\includegraphics[scale = 0.6]
    {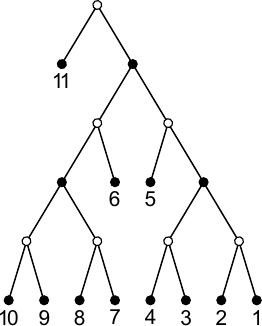}
    &
    \quad\includegraphics[scale = 0.6]{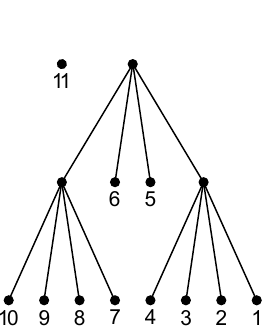}
    \\
    \includegraphics[scale = 0.6]{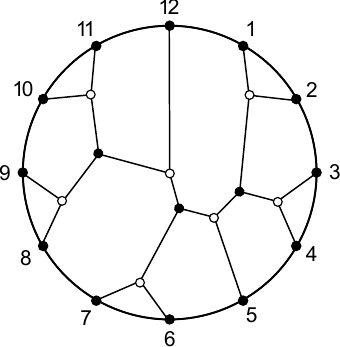}
    &
    \quad\includegraphics[scale = 0.6]
    {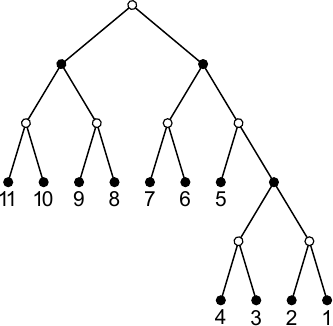}
    &
    \quad\includegraphics[scale = 0.6]{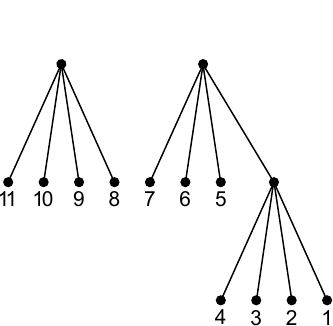}
    \end{tabular}
    
    \caption{Two examples of the bijection in the proof of Proposition \ref{prop:counting_trees}, between webs $X\in\mathscr{T}_{3,12}^\text{std}$ (left), binary trees $B$ (center), and ordered pairs of 4-ary trees $Q$ (right).}
    \label{fig:webs&trees}
\end{figure}

\newpage

\appendix

\section{Dual web bases for $\W_\lambda(\C^3)$ and $\W_\lambda(\C^4)$ when $\lambda=(1^{12})$}

Here we illustrate standard $\SL_4$ webs and $\SL_3$ webs when $n=12$.  These are each grouped together in their rotation and reflection equivalence classes and are paired as dual bases (see Theorem \ref{thm:DualityComputation}) with respect to the FLL pairing, with the Yamanouchi words of their associated standard Young tableaux included.


 
\NewColumnType{N}{@{}Q[t,c,4pt]}  
\NewColumnType{W}{@{}Q[t,c,.16\textwidth]@{}}

\newcommand{\num}[1]{\adjustbox{valign=t}{#1.}}

\newcommand{\web}[2]{\adjustbox{valign=t}{\includegraphics[width=.15\textwidth]{webs/_#1,12_web#2.pdf}}}

\newcommand{\word}[1]{\adjustbox{raise=-10pt}{#1}}

\begin{table}[H]
\begin{center}
\small{

\begin{tblr}{colspec={ | N W W | N W W | N W W | }}
\hline

\num{1} 
&{\web{4}{1} \\ \word{111222333444}}
&{\web{3}{1} \\ \word{123123123123}}
&\num{2}
&{\web{4}{2} \\ \word{111222334344}}
&{\web{3}{2} \\ \word{123123121323}}
&\num{3}
&{\web{4}{4} \\ \word{111222334434}}
&{\web{3}{4} \\ \word{123123121233}}
\\
\hline

\num{4}
&{\web{4}{5} \\ \word{111222343434}}
&{\web{3}{5} \\ \word{123123112233}}
&\num{5}
&{\web{4}{6} \\ \word{111223234344}}
&{\web{3}{6} \\ \word{123121321323}}
&\num{6}
&{\web{4}{18} \\ \word{111223234434}}
&{\web{3}{7} \\ \word{123121321233}}
\\
\hline

\num{7}
&{\web{4}{8} \\ \word{111223243344}}
&{\web{3}{8} \\ \word{123121312323}}
&\num{8}
&{\web{4}{20} \\ \word{111223243434}}
&{\web{3}{20} \\ \word{123121312233}}
&\num{9}
&{\web{4}{13} \\ \word{111223342344}}
&{\web{3}{13} \\ \word{123121213323}}
\\ 
\hline

\num{10}
&{\web{4}{24} \\ \word{111223342434}}
&{\web{3}{10} \\ \word{123121213233}}
&\num{11}
&{\web{4}{12} \\ \word{111223344234}}
&{\web{3}{12} \\ \word{123121212333}}
&\num{12}
&{\web{4}{21} \\ \word{111223423434}}
&{\web{3}{21} \\ \word{123121132233}}
\\
\hline

\num{13}
&{\web{4}{15} \\ \word{111223432344}}
&{\web{3}{15} \\ \word{123121123323}}
&\num{14}
&{\web{4}{26} \\ \word{111223432434}}
&{\web{3}{26} \\ \word{123121123233}}
&\num{15}
&{\web{4}{28} \\ \word{111223434234}}
&{\web{3}{28} \\ \word{123121122333}}
\\
\hline

\num{16}
&\SetCell[c=7]{}
{\adjustbox{valign=t}{$\adjustbox{valign=c}{\includegraphics[width=.15\textwidth]{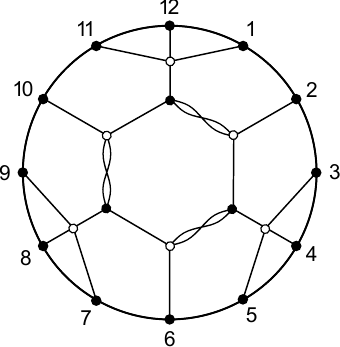}} -  \adjustbox{valign=c}{\includegraphics[width=.15\textwidth]{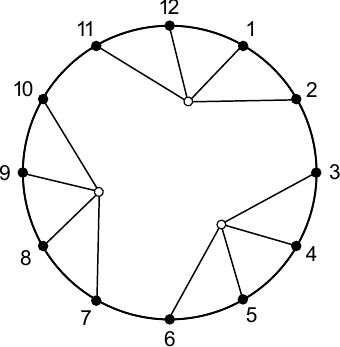}} =  
\adjustbox{valign=c}{\includegraphics[width=.15\textwidth]{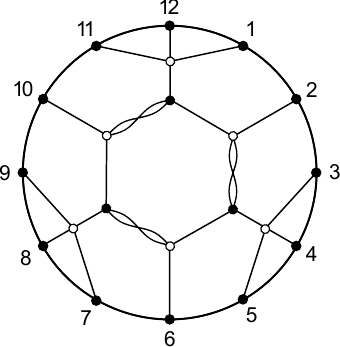}} -  \adjustbox{valign=c}{\includegraphics[width=.15\textwidth]{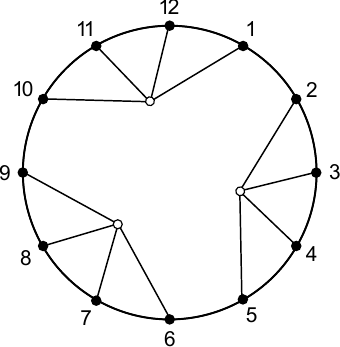}}$} \\ \word{111232234434}}
&&&&&&& {\web{3}{19} \\ \word{123112321233}}
\\
\hline

\end{tblr}
}
\end{center}
\caption{Dual bases for $\W_\lambda(\C^4)$ and $\W_\lambda(\C^3)$, $\lambda=(1,\ldots,1)$. The cell numbered $i$ depicts webs yielding dual basis elements $\mathbf{W}_i \in \W_\lambda(\C^4)$ (left) and $\mathbf{X}_i\in\W_\lambda(\C^3)$ (right), as well as the Yamanouchi words of their associated standard Young tableaux.}
\label{tablep1}
\end{table}

\begin{table}[H]
\begin{center}
\small{

\begin{tblr}{vline{7}={1-5}{solid}, vline{Z}={1-5}{solid},
             colspec={ | N W W | N W W  N W W  }}
\hline

\num{17}
&{\web{4}{25} \\ \word{111232342434}}
&{\web{3}{25} \\ \word{123112213233}}
&\num{18}
&{\web{4}{23} \\ \word{111232344234}}
&{\web{3}{23} \\ \word{123112212333}}
&\num{19}
&{\web{4}{29} \\ \word{111232434234}}
&{\web{3}{29} \\ \word{123112122333}}
\\
\hline

\num{20}
&{\web{4}{30} \\ \word{111234234234}}
&{\web{3}{30} \\ \word{123111222333}}
&\num{21}
&{\web{4}{31} \\ \word{112122334344}}
&{\web{3}{31} \\ \word{121323121323}}
&\num{22}
&{\web{4}{32} \\ \word{112122343434}}
&{\web{3}{32} \\ \word{121323112233}}
\\
\hline

\num{23}
&{\web{4}{33} \\ \word{112123234344}}
&{\web{3}{33} \\ \word{121321321323}}
&\num{24}
&{\web{4}{39} \\ \word{112123243434}}
&{\web{3}{39} \\ \word{121321312233}}
&\num{25}
&{\web{4}{35} \\ \word{112123423344}}
&{\web{3}{35} \\ \word{121321132323}}
\\
\hline

\num{26}
&{\web{4}{40} \\ \word{112123423434}}
&{\web{3}{40} \\ \word{121321132233}}
&\num{27}
&{\web{4}{37} \\ \word{112123434234}}
&{\web{3}{37} \\ \word{121321122333}}
&\num{28}
&{\web{4}{44} \\ \word{112132324344}}
&{\web{3}{44} \\ \word{121312231323}}
\\
\hline

\num{29}
&{\web{4}{38} \\ \word{112132342344}}
&{\web{3}{38} \\ \word{121312213323}}
&\num{30}
&{\web{4}{41} \\ \word{112234123434}}
&{\web{3}{41} \\ \word{121211332233}}
&\num{31}
&{\web{4}{42} \\ \word{112312423434}}
&{\web{3}{42} \\ \word{121132132233}}
\\
\hline

\num{32}
&{\web{4}{43} \\ \word{112341234234}}
&{\web{3}{43} \\ \word{121113222333}}
\\
\cline{1-3}

\end{tblr}
}
\end{center}
\caption{Continued from Table \ref{tablep1}.}
\label{tablep2}
\end{table}


\newpage
\section{Arborizable indecomposable webs for $\C[\widehat{\normalfont{\text{Gr}}}(3,11)]$, $\C[\widehat{\normalfont{\text{Gr}}}(3,10)]$, and $\C[\widehat{\normalfont{\text{Gr}}}(3,9)]$ }
\label{appendix:clasps}

Starting from the webs for $\C[\Gr(3,12)]$ (in Tables \ref{tablep1} and \ref{tablep2}), we iteratively apply clasping to get webs for $\C[\Gr(3,n)]$ for $n=11, 10, 9$, and check which of these are arborizable, indecomposable, and non-elliptic.  We then record how many are in each dihedral orbit.

\renewcommand{\web}[1]{\adjustbox{valign=c}{\includegraphics[width=.145\textwidth]{webs/#1.pdf}}}

\begin{table}[H]
\begin{tabular}{c | l | c}

Unclasped web & \begin{centering} Clasped webs for $\C[\Gr(3,11)]$ \end{centering}& \#
\\
\hline

\web{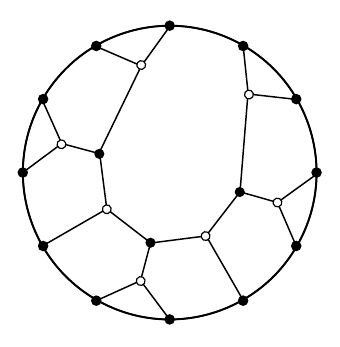}
&\web{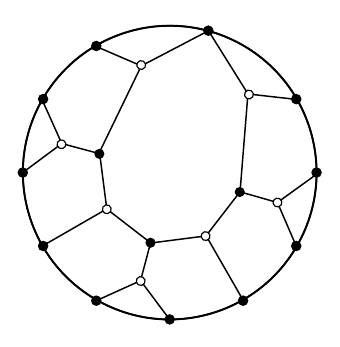}
&11
\\
\hline

\web{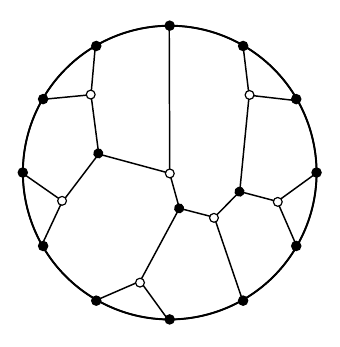}
&\web{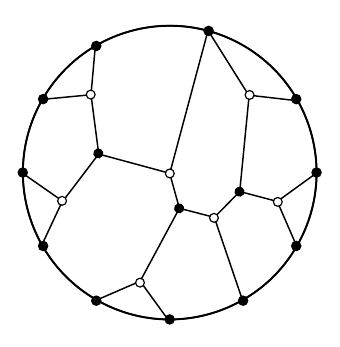}
\web{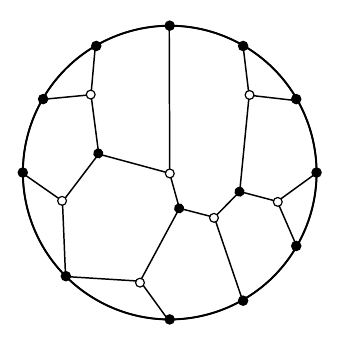}
&44
\\
\hline

\web{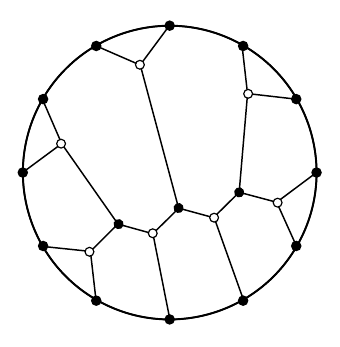}
&\web{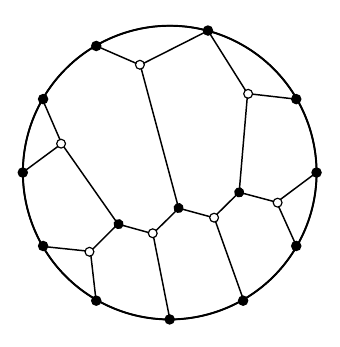}
&22
\\
\hline

\web{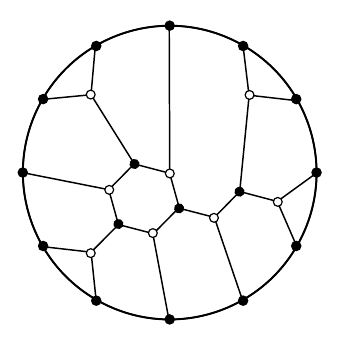}
&\web{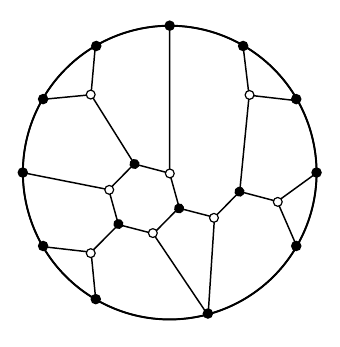}
\web{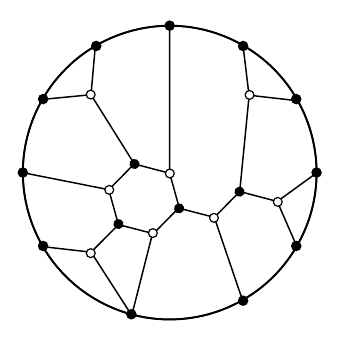}
\web{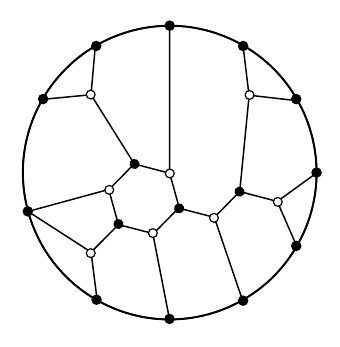}
\web{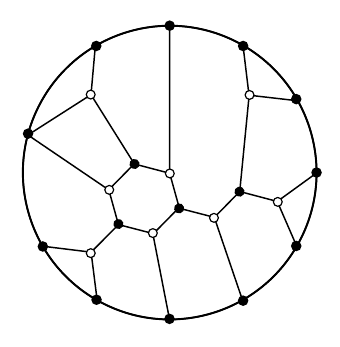}
\web{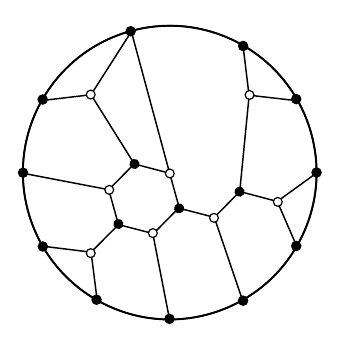}
&110
\\
\hline

\web{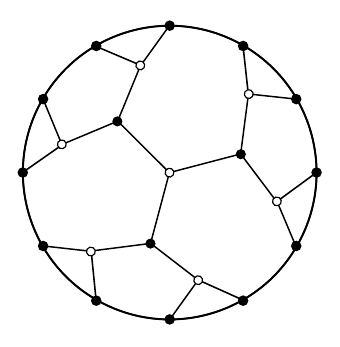}
&\web{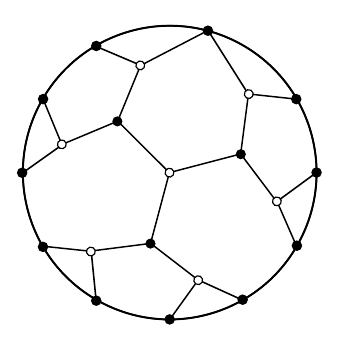}
&11
\\
\hline

\web{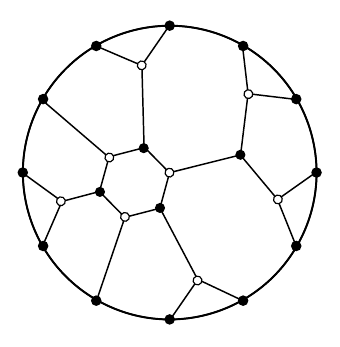}
&\web{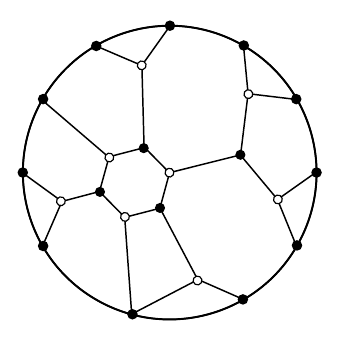}
\web{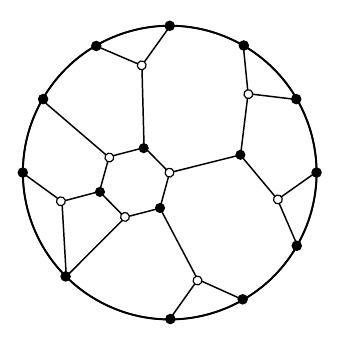}
&44
\\
\hline

\web{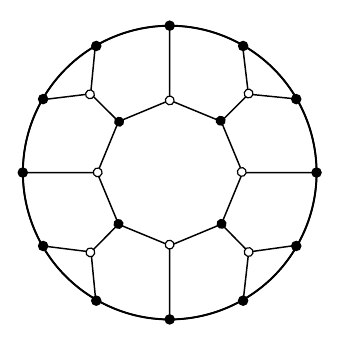}
&\web{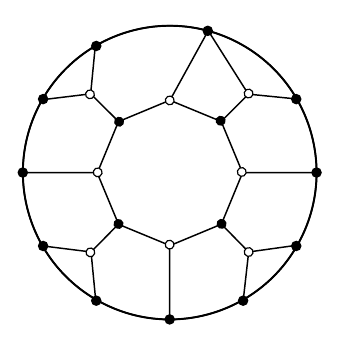}
&22
\\
\hline

\end{tabular}

\caption{All 264 arborizable, indecomposable, non-elliptic semistandard $\SL_3$ webs of Pl\"ucker degree 4 on 11 boundary vertices, up to rotations and reflections.}
\label{table:Gr(3,11)clvars}
\end{table}

\begin{table}[H]
\begin{tabular}{c | l | c}

Unclasped web & Clasped webs for $\C[\Gr(3,10)]$ & \#
\\
\hline

\web{_3,12_web39nonums}
&\web{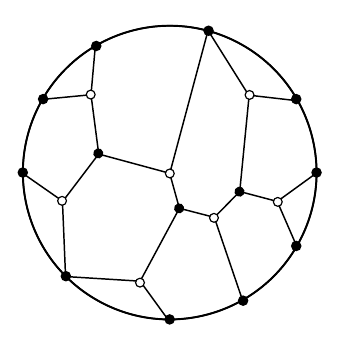}
&20
\\
\hline

\web{_3,12_web35nonums}
&\web{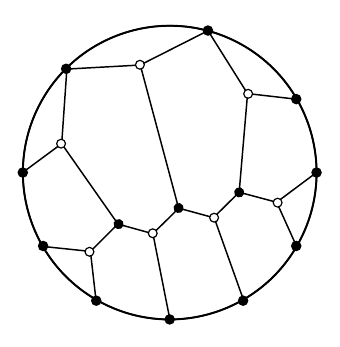}
&10
\\
\hline

\web{_3,12_web40nonums}
&\web{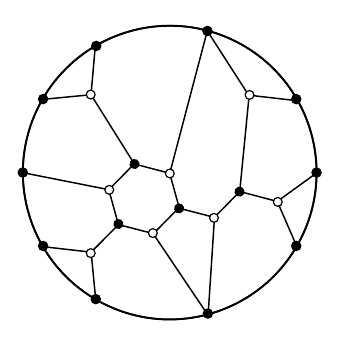}
\web{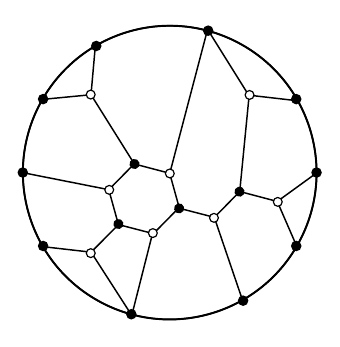}
\web{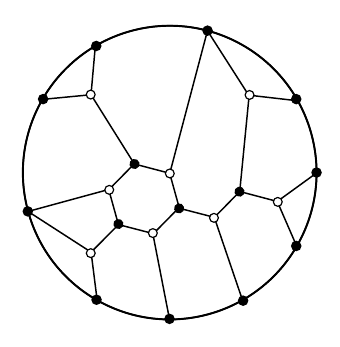}
\web{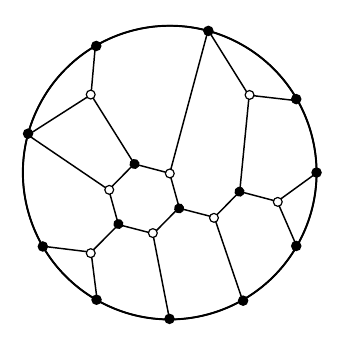}
\web{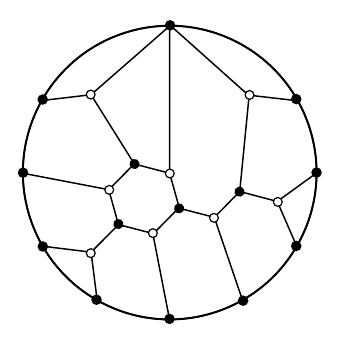}
&100
\\
\hline

\web{_3,12_web44nonums}
&\web{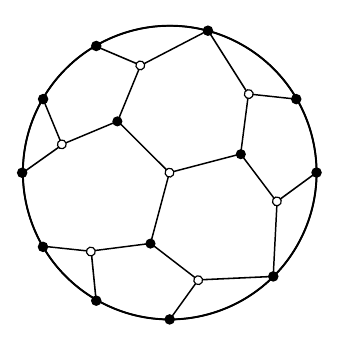}
&10
\\
\hline

\web{_3,12_web38nonums}
&\web{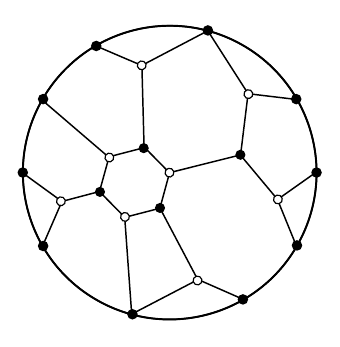}
\web{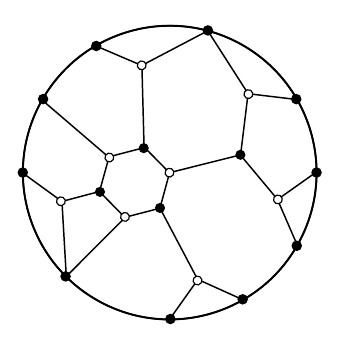}
\web{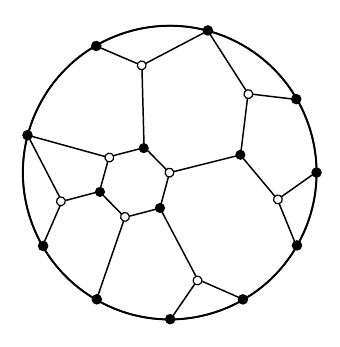}
\web{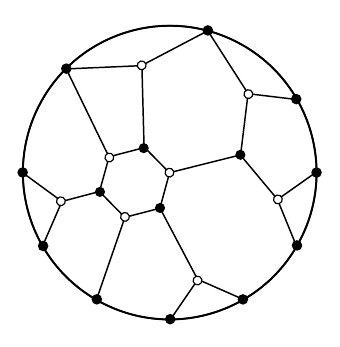}
&80
\\
\hline

\web{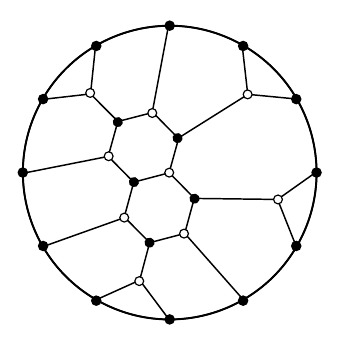}
&\web{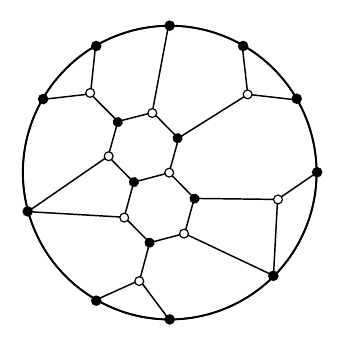}
\web{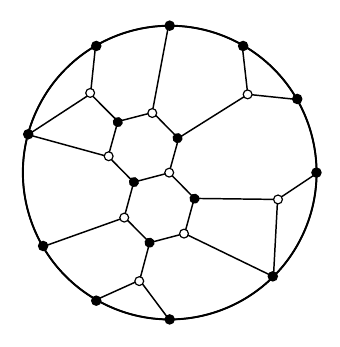}
\web{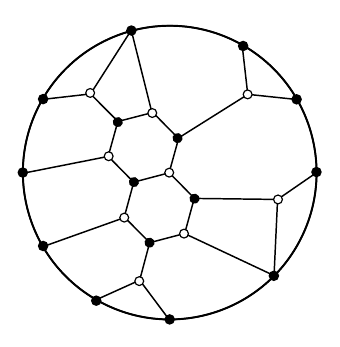}
\web{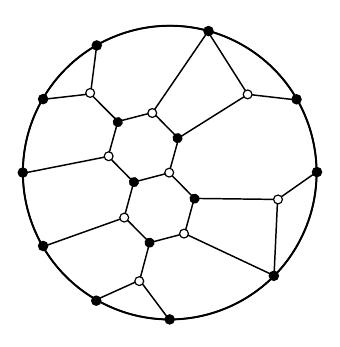}
\web{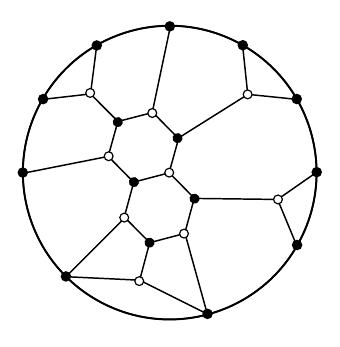}
&150
\\
&\web{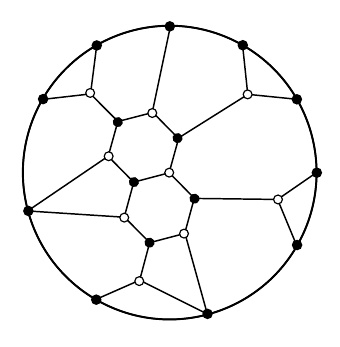}
\web{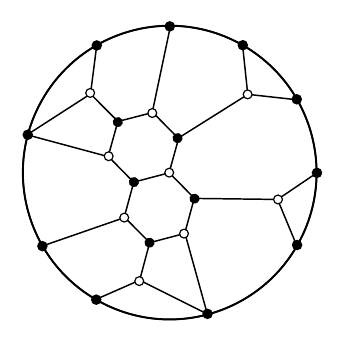}
\web{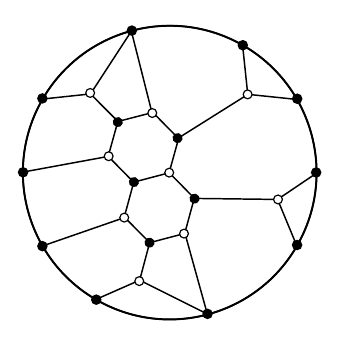}
\web{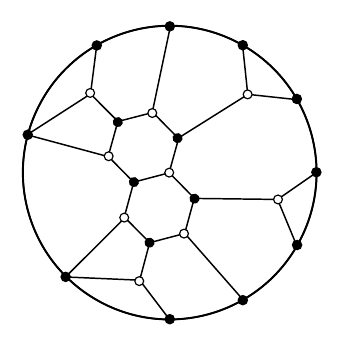}
\\
\hline

\web{_3,12_web42nonums}
&\web{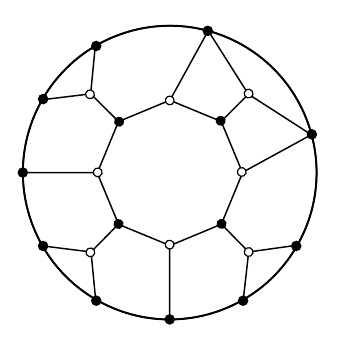}
&10
\\
\hline

\web{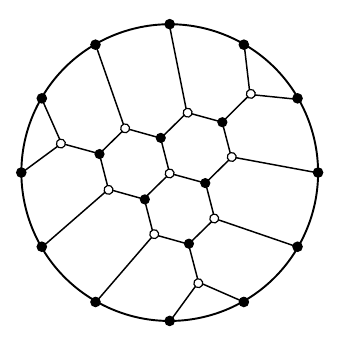}
&\web{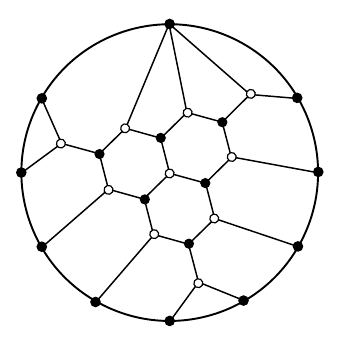}
&20
\\
\hline
\end{tabular}

\caption{All 400 arborizable, indecomposable, non-elliptic semistandard $\SL_3$ webs of Pl\"ucker degree 4 on 10 boundary vertices, up to rotations and reflections.}
\label{table:Gr(3,10)clvars}
\end{table}


\begin{table}[H]
\begin{tabular}{c | l | c}

Unclasped web & Clasped webs for $\C[\Gr(3,9)]$ & \#
\\
\hline

\web{_3,12_web44nonums}
&\web{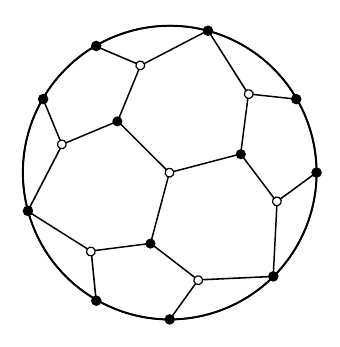}
&3
\\
\hline

\web{_3,12_web38nonums}
&\web{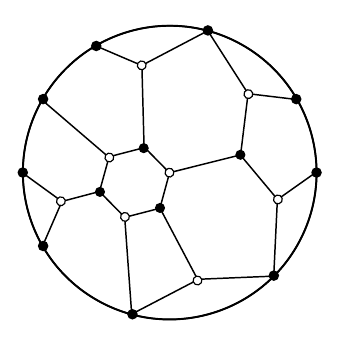}
\web{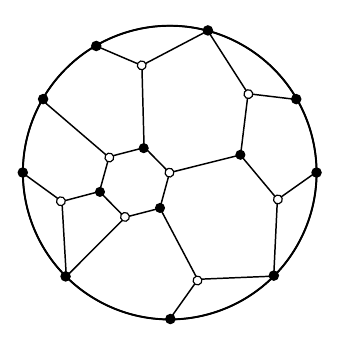}
&36
\\
\hline

\web{_3,12_web41nonums}
&\web{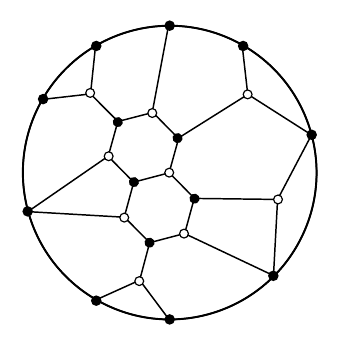}
\web{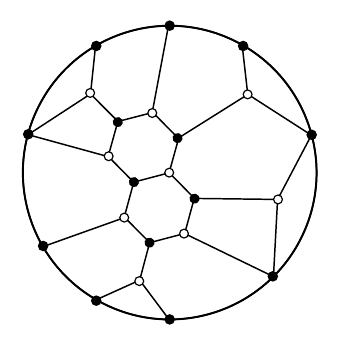}
\web{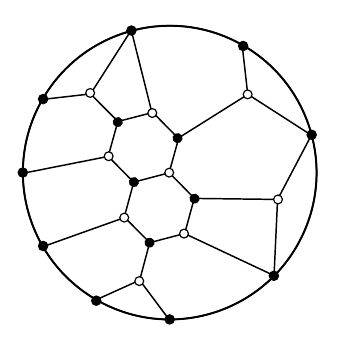}
\web{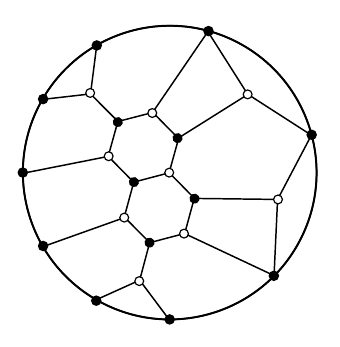}
\web{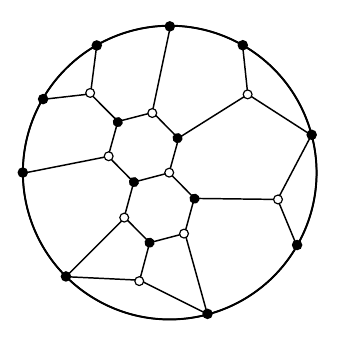}
&135
\\
&\web{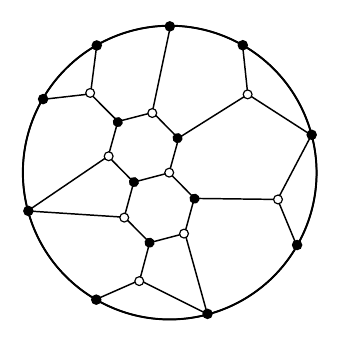}
\web{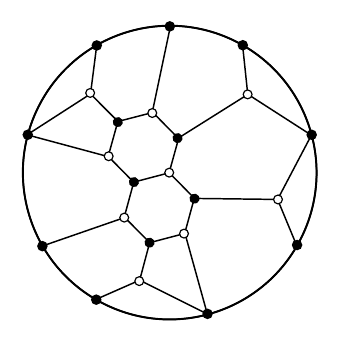}
\web{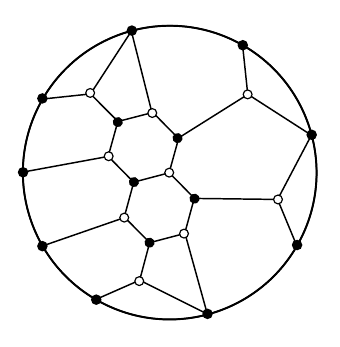}
\web{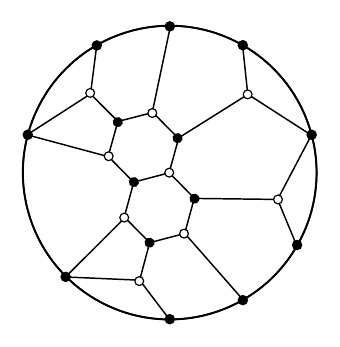}
\\
\hline

\web{_3,12_web43nonums}
&\web{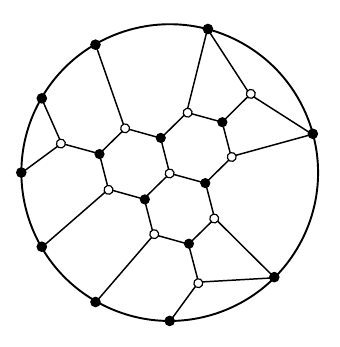}
\web{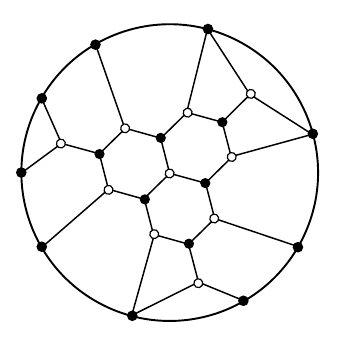}
\web{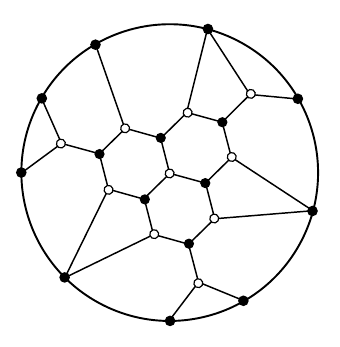}
\web{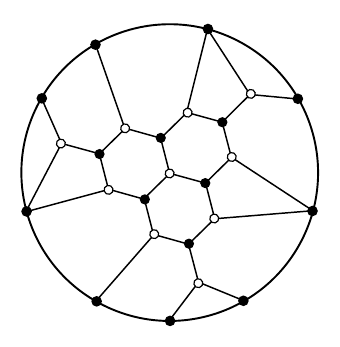}
\web{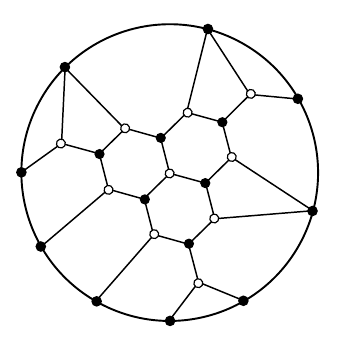}
&114
\\
&\web{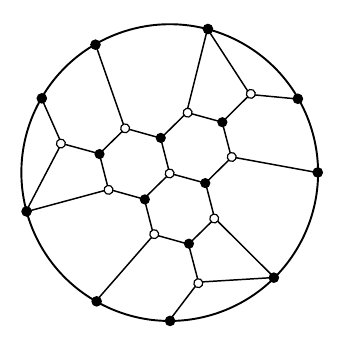}
\web{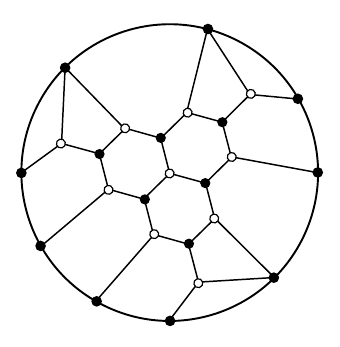}
\\
\hline
\end{tabular}

\caption{All 288 arborizable, indecomposable, non-elliptic semistandard $\SL_3$ webs of Pl\"ucker degree 4 on 9 boundary vertices, up to rotations and reflections.}
\label{table:Gr(3,9)clvars}
\end{table}

\newpage
\section*{Acknowledgments}
We are grateful to Ellis Caird, Alastair King, Jianrong Li, Pasha Pylyavskyy, Melissa Sherman-Bennett, and David Speyer for helpful conversations. We are especially thankful to Chris Fraser for his careful reading of an earlier version of this work, and his suggested interpretation of our results in terms of a twisted boundary measurement map.  We also thank the referees for their  helpful comments that improved the exposition.

E.C. was partially supported by a GAANN Fellowship under Award P200A240046. C.G. was partially supported by NSF grant DMS-2452032 and by a travel grant from the Simons Foundation. K.W. was supported by NSF grant DMS-2039316 through the University of Oregon. G.M. was partially funded by Simons Foundation Travel Support for Mathematicians.  We appreciate the MRWAC 2024 event for bringing our group together, supported by NSF grants DMS-MS-1745638 and DMS-1854162.

\bibliographystyle{plain}
\bibliography{arxiv-v2}
\end{document}